\documentclass[11pt]{article}
\usepackage[a4paper]{geometry}
\usepackage{graphicx}

\usepackage[colorlinks=true,citecolor=black,linkcolor=black,urlcolor=blue]{hyperref}

\usepackage{amsmath,amsthm,amssymb,amsfonts}
\usepackage{color}
\usepackage{caption}
\usepackage{subcaption}%
\newcommand{\arxiv}[1]{\href{http://arxiv.org/abs/#1}{\texttt{arXiv:#1}}}

\theoremstyle{plain}
\newtheorem{theorem}{Theorem}
\newtheorem{lemma}[theorem]{Lemma}
\newtheorem{corollary}[theorem]{Corollary}
\newtheorem{proposition}[theorem]{Proposition}

\newtheorem{observation}[theorem]{Observation}

\theoremstyle{definition}
\newtheorem{definition}[theorem]{Definition}

\newtheorem{question}[theorem]{Question}

\theoremstyle{remark}
\newtheorem{remark}[theorem]{Remark}

\makeatletter
\newtheorem*{rep@theorem}{\rep@title}
\newcommand{\newreptheorem}[2]{%
\newenvironment{rep#1}[1]{%
 \def\rep@title{#2 \ref{##1}}%
 \begin{rep@theorem}}%
 {\end{rep@theorem}}}
\makeatother

\newreptheorem{theorem}{Theorem}

\newcommand{\ie}{i.e.,~} %

\def\EE{\mathbb{E}}

\def\RR{\mathbb{R}}
\def\PP{\mathbb{P}}
\def\RR{\mathbb{R}}

\def\ZZ{\mathbb{Z}}

\newcommand{\trans}[1]{{#1}^\top}

\newcommand{\pol}[1]{{#1}^\circ}
\DeclareMathOperator{\Vol}{Vol}
\DeclareMathOperator{\conv}{conv}

\newcommand{\li}[2]{\ell_{#1,#2}} %

\DeclareMathOperator{\pr}{rank_+}
\DeclareMathOperator{\ec}{xc}
\DeclareMathOperator{\ic}{ic}

\newcommand{\sprod}[2]{\langle {#1} , {#2} \rangle} %
\newcommand{\defn}[1]{\emph{\color{blue} #1}} %
\newcommand{\set}[2]{\ensuremath{\left\{#1\,\middle|\,#2\right\}}} 
\newcommand{\smallset}[2]{\ensuremath{\{#1\mid#2\}}} 
\newcommand{\ffloor}[2]{\left\lfloor{\frac{#1}{#2}}\right\rfloor} %
\newcommand{\fceil}[2]{\left\lceil {\frac{#1}{#2}} \right\rceil} %

\makeatletter
\let\@fnsymbol\@arabic
\makeatother

\title{\bf Polygons as sections
of higher-dimensional polytopes}

\author{Arnau Padrol\thanks{Supported by the DFG
Collaborative Research Center SFB/TR~109 ``Discretization
in Geometry and Dynamics''.}\\
\small Institut f\"ur Mathematik\\[-0.8ex]
\small Freie Universit\"at Berlin\\[-0.8ex] 
\small Berlin, Germany\\
\small\tt arnau.padrol@fu-berlin.de\\
\and
Julian Pfeifle\thanks{Supported by grants COMPOSE
  (EUI-EURC-2011-4306), MTM~2012-30951, MTM~2011-24097 and 2009-SGR-1040.} \\
\small Dept.\ Matem\`atica Aplicada II\\[-0.8ex]
\small Universitat Polit\`ecnica de Catalunya\\[-0.8ex]
\small Barcelona, Spain\\
\small\tt julian.pfeifle@upc.edu
}

\date{}

\begin{document}

\maketitle

\begin{abstract}
We show that every heptagon is a section of a $3$-polytope with $6$ vertices.
This implies that every $n$-gon with $n\geq 7$ can be obtained as a section of a
$(2+\ffloor{n}{7})$-dimensional polytope with at most $\fceil{6n}{7}$ vertices; and
 provides a geometric proof of the fact that every nonnegative
$n\times m$ matrix of rank~$3$ has nonnegative
rank not larger than $\fceil{6\min(n,m)}{7}$. This result has been independently
proved, algebraically, by Shitov (J. Combin. Theory Ser. A 122,
2014).

\end{abstract}

\section{Introduction}

Let $P$ be a (convex) polytope. 
An \defn{extension} of~$P$ is any polytope~$Q$ such that~$P$ is the image of $Q$ under a linear projection;
 the \defn{extension
complexity} of~$P$, denoted~\defn{$\ec(P)$}, is the minimal number of facets 
of an extension of~$P$. This concept is relevant in combinatorial optimization
because if a polytope has low extension complexity, then it is possible to
use an extension with few facets to efficiently optimize a linear
functional over it.

A \defn{section} of a polytope is its intersection with an affine subspace.
We will work with the polar formulation of the problem above, which
asks for the minimal number of vertices of a polytope $Q$ that has $P$ as a section.
If we call this quantity the \defn{intersection complexity} of $P$, \defn{$\ic(P)$}, 
then by definition it holds that $\ic(P)=\ec(\pol P)$, where $\pol P$ is the polar dual of~$P$.
However, extension complexity is preserved under polarity 
(see~\cite[Proposition~2.8]{ThomasParriloGouveia2013}),
so these four quantities actually coincide:
\[\ic(P)=\ec(\pol P)=\ec(P)=\ic(\pol P).\]

Despite the increasing amount of attention that this topic has received recently
(see \cite{FioriniRothvosTiwary2012}, \cite{ThomasParriloGouveia2013}, \cite{GouveiaRobinsonThomas2013}, \cite{Shitov2014} and references therein),
it is still far from being well understood. For example, even the possible range of values of the intersection complexity of an $n$-gon is still unknown.
Obviously, every $n$-gon has intersection complexity at most~$n$, and for those with 
$n\leq 5$ it is indeed exactly~$n$. It is not hard to
check that hexagons can have complexity $5$ or $6$
(cf. \cite[Example~3.4]{ThomasParriloGouveia2013}) and, as we show in Proposition~\ref{prop:ichexagon},
it is easy to decide which is the exact value.

By proving that a certain pseudo-line arrangement is not stretchable, we show that
every heptagon is a section of a $3$-polytope with no more than $6$ vertices.
This reveals the geometry behind a result found independently by 
Shitov in~\cite{Shitov2014}, 
and allows us to settle the intersection complexity of heptagons.

\begin{reptheorem}{thm:icheptagon}
Every heptagon has intersection complexity $6$.
\end{reptheorem}

In general, the minimal intersection complexity of an $n$-gon is $\theta(\log n)$,
which is attained by regular $n$-gons \cite{BenTalNemirovski2001}\cite{FioriniRothvosTiwary2012}.
On the other hand, there exist $n$-gons whose intersection complexity is at least $\sqrt{2n}$~\cite{FioriniRothvosTiwary2012}.
As a consequence of Theorem~\ref{thm:icheptagon} we automatically get upper bounds for the
complexity of arbitrary $n$-gons.

\begin{reptheorem}{thm:icngon}
Any $n$-gon with $n\geq 7$ is a section of a $(2+\ffloor{n}{7})$-dimensional
polytope with at most $\fceil{6n}{7}$ vertices. In particular, $\ic(P)\leq
\fceil{6n}{7}$.
\end{reptheorem}

Of course, this is just a first step towards understanding the
intersection complexity of polygons. By counting degrees of freedom, it is
conceivable that every $n$-gon could be represented as a section of an $O(\sqrt{n})$-dimensional polytope
with $O(\sqrt{n})$ vertices. For sections of $3$-polytopes, our result only
shows that every $n$-gon is a section of a $3$-polytope with not more than $n-1$ vertices, whereas
we could expect an order of $\frac23 n$ vertices.

There is an alternative formulation of these results. The \defn{nonnegative rank} of 
a nonnegative $n\times m$ matrix $M$,
denoted \defn{$\pr(M)$}, is the minimal number~$r$ such that there
exists an $n\times r$ nonnegative matrix~$R$ and an $r\times m$ nonnegative matrix
$S$ such that \(M=RS.\) 
A classical result of Yannakakis~\cite{Yannakakis1991}
states that the intersection complexity of a polytope coincides with the
{nonnegative rank} of its slack matrix. 
In this setting, it is not hard to
deduce the following theorem from Theorem~\ref{thm:icngon} (it is easy to
deal with matrices of rank~$3$ that are not slack
matrices).

\begin{theorem}[{{\cite[Theorem~3.2]{Shitov2014}}}]\label{thm:nonnegativerank}
 Let $M$ be a nonnegative $n\times m$ matrix of rank~$3$. %
 Then $\pr(M)\leq \fceil{6\min(n,m)}{7}$.
\end{theorem}

This disproved a conjecture of
Beasley and Laffey (originally stated in \cite[Conjecture~3.2]{BeasleyLaffey2009} in a more
general setting), who asked if for any $n\geq 3$ there is an $n\times n$ nonnegative matrix $M$ of rank~$3$ with $\pr(M)=n$. While this paper was under review, Shitov improved Theorem~\ref{thm:icngon} and provided a sublinear upper bound for the intersection/extension complexity of $n$-gons~\cite{Shitov2014b}.

\subsection{Notation}
We assume throughout that the vertices $\set{p_i}{i\in \ZZ/n\ZZ}$
of every $n$-gon $P$ are cyclically clockwise labeled, \ie the edges of $P$
are $\conv\{p_i,p_{i+1}\}$ for $i\in\ZZ/n\ZZ$
and the triangles $(p_{i+2},p_{i+1},p_i)$ are positively oriented for $i\in\ZZ/n\ZZ$.

We regard the $p_i$ as points of the Euclidean plane $\EE^2$, embedded in the
real projective plane $\PP^2$ as $\EE^2 = 
\smallset{\trans{(x, y, 1)}}{ x, y \in\RR}$. For any pair of points $p, q \in
\EE^2$, we denote by $\li{p}{q} = p \wedge q$ the line joining them. It is well
known that $\li{p}{q}$ can be identified with the point $\li{p}{q} = p \times q$
in the dual space $(\PP^2)^*$, where
\[p\times q =\trans{\left( \begin{vmatrix}p_2& p_3\\q_2&q_3\end{vmatrix},\; -\begin{vmatrix}p_1& p_3\\q_1&q_3\end{vmatrix},\; \begin{vmatrix}p_1& p_2\\q_1&q_2\end{vmatrix} \right)}\]
 denotes the cross-%
 product in Euclidean $3$-space. Similarly, the meet
$\ell_1\vee\ell_2$ of two lines $\ell_1,\ell_2\in (\PP^2)^*$ is their
intersection point in $\PP^2$, which has coordinates~$\ell_1\times\ell_2$.

\section{The intersection complexity of hexagons}
As an introduction for the techniques that we use later with heptagons, we study the intersection complexity of hexagons. Hexagons can have intersection
complexity either~$5$ or $6$~\cite[Example~3.4]{ThomasParriloGouveia2013}. In this section we provide a geometric condition to decide among the two values. 
This section is mostly independent from the next two, and the reader can safely skip it.

First, we introduce a lower bound for the $3$-dimensional intersection complexity of $n$-gons that we will use later.

\begin{proposition}\label{prop:ic3bound}
No $n$-gon can be obtained as a section of a $3$-polytope with less than $\fceil{n+4}{2}$ vertices. 
\end{proposition}
\begin{proof}
 Let $Q$ be a $3$-polytope with $m$ vertices such that its intersection with the 
 plane~$H$ coincides with $P$, and let $k$ be the number of vertices of $Q$ that lie on $H$.
 
 By Euler's formula, the number of edges of $Q$ is at most~$3m-6$, of which at least $3k$ have an endpoint on $H$. Moreover, the subgraphs $G^+$ and $G^-$ consisting of edges of $Q$ lying in the open halfspaces $H^+$ and $H^-$ are both connected. Indeed, if $H=\set{x}{\sprod{a}{x}=b}$, then the linear function $\sprod{a}{x}$ induces an acyclic partial orientation on $G^+$ and $G^-$ by setting $v\rightarrow w$ when $\sprod{a}{v}<\sprod{a}{w}$. Following this orientation we can connect each vertex of $G^+$ to the face of $Q$ that maximizes $\sprod{a}{x}$, and following the reverse orientation, each vertex of $G^-$ to the face that minimizes $\sprod{a}{x}$ (compare~\cite[Theorem~3.14]{Ziegler1995}). 
 
 Hence, there are at least $m-k-2$ edges in $G^+\cup G^-$. These are edges of $Q$ that do not intersect $H$. There are also at least $3k$ edges that have an endpoint on~$H$. Now, observe that every vertex of $P$ is either a vertex of $Q$ in $H$ or is the intersection with~$H$ of an edge of $Q$ that has an endpoint at each side of~$H$. Hence,
 \begin{equation}\label{eq:boundnumvertices}n-k\leq (3m-6)-(3k)-(m-k-2)=2m-4-2k,\end{equation}
 and since $k\geq 0$, we get the desired bound.
 \end{proof}

The lower bound of Proposition~\ref{prop:ic3bound} is optimal: for every $m\geq 2$ there are $2m$-gons appearing as sections of $3$-polytopes with $m+2$ vertices (Figure~\ref{fig:optimalcuts}).

\begin{figure}[htpb]
\centering
\includegraphics[width=.22\textwidth]{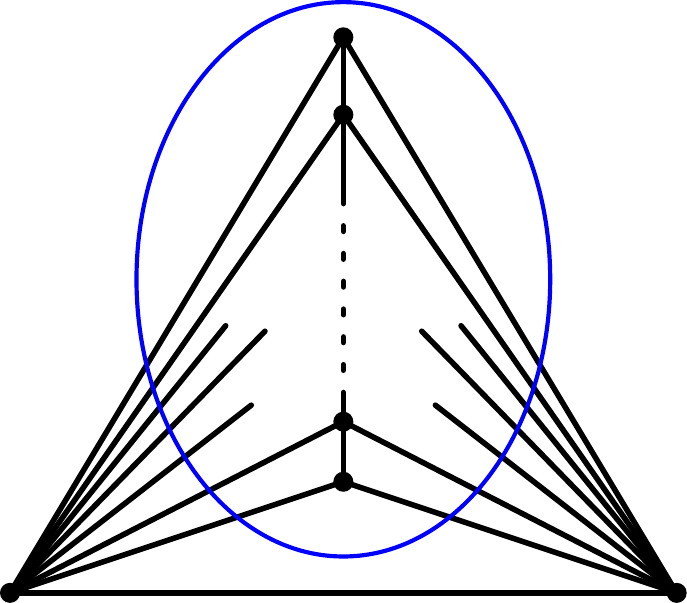}
\caption{For $m\geq 2$, the join of an $m$-path with an edge is the graph of a stacked $3$-polytope with $2m+2$ vertices that has a $2m$-gon as a section (by a plane that truncates the edge).}\label{fig:optimalcuts}
\end{figure}

\begin{corollary}
 The intersection complexity of a hexagon is either $5$ or $6$.
\end{corollary}

\begin{proposition}\label{prop:ichexagon}
 The intersection complexity of a hexagon is $5$ if and only if the lines
$\li{p_0}{p_5}$, $\li{p_1}{p_4}$ and $\li{p_2}{p_3}$ intersect in a common point
of the projective plane for some cyclic labeling of its vertices $\set{p_i}{i\in\ZZ/6\ZZ}$.
\end{proposition}

\begin{figure}[htpb]
\centering
\subcaptionbox{Non-hexagonal cuts.}[
.58\textwidth ]
{\includegraphics[width=.55\textwidth]{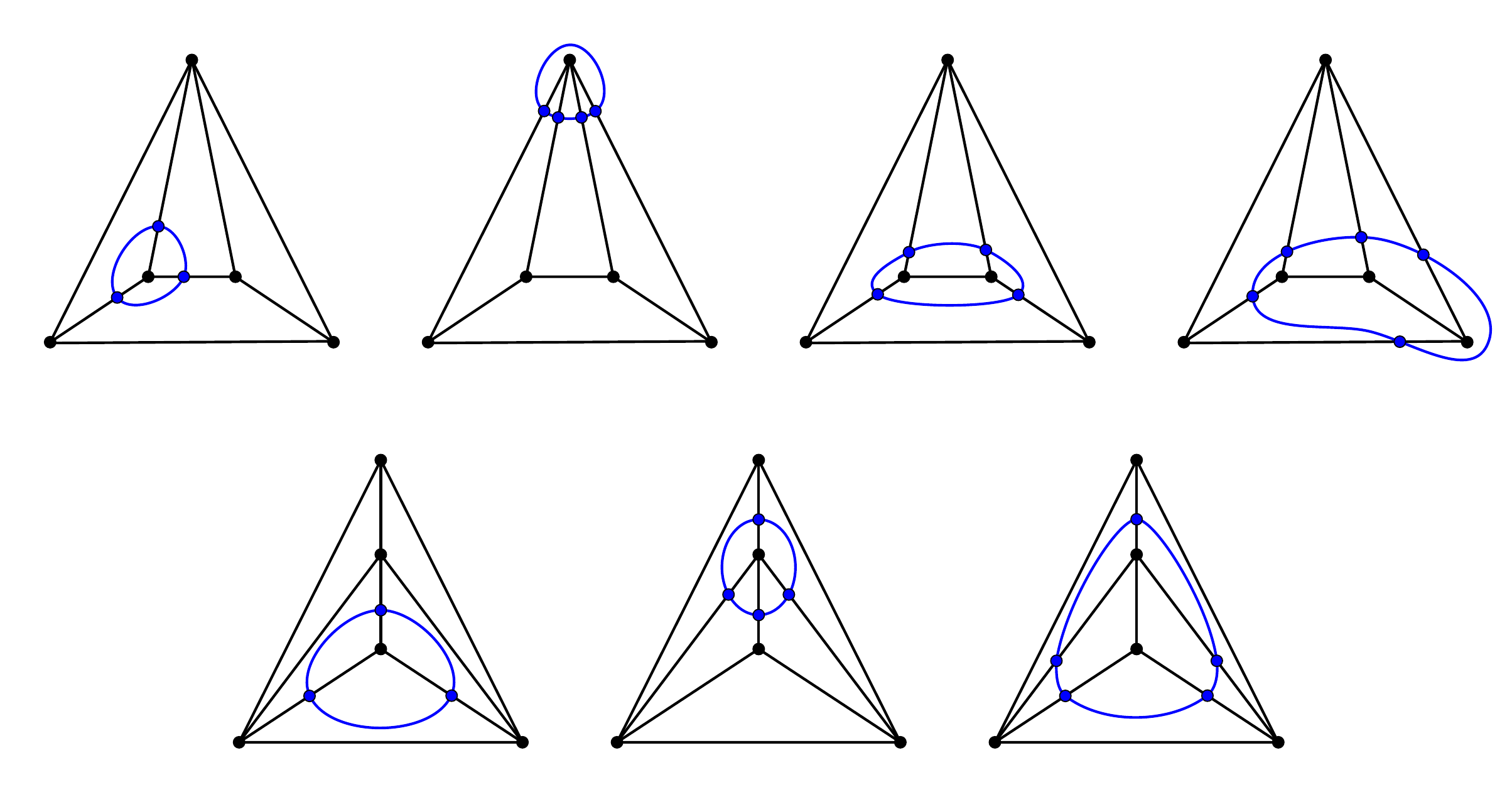}}
 \subcaptionbox{The only hexagonal cut.\label{fig:bipyrcut}}[.4\textwidth ]
{\includegraphics[width=.3\textwidth]{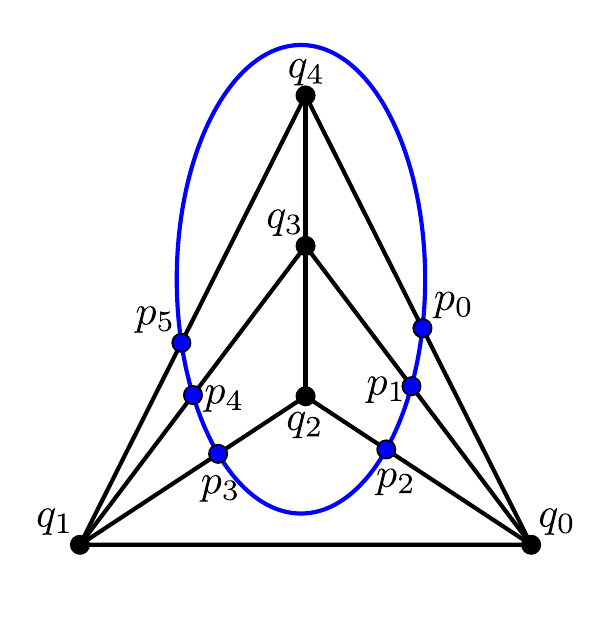}}
\caption{All cuts of the quadrangular pyramid and the triangular bipyramid into two connected components, up to symmetry.}\label{fig:cuts}
\end{figure}

\begin{proof}
 The only $4$-polytope with $5$ vertices is the simplex, which only has $5$ facets;
thus, none of its $2$-dimensional sections is a hexagon. Therefore, if $P$ is a hexagon, then 
 $\ic(P)=5$ if and only if it is the intersection of a $2$-plane~$H$ with a $3$-polytope~$Q$ with $5$~vertices.

 There are only two combinatorial types of $3$-polytopes with $5$~vertices: the quadrangular pyramid and the triangular bipyramid. By \eqref{eq:boundnumvertices}, $H$ does not contain any vertex of~$Q$. Hence, $H$ induces a cut of the graph of $Q$ into two (nonempty) disjoint connected components.
 A small case-by-case analysis (cf. Figure~\ref{fig:cuts}) tells us that the only 
 possibility is that $Q$ is the bipyramid and $H$ cuts its graph as shown in Figure~\ref{fig:bipyrcut}. However, in every geometric realization of such a cut (with the same labeling), the lines $\li{p_0}{p_5}$, $\li{p_1}{p_4}$ and $\li{p_2}{p_3}$ intersect in a common (projective) point: the point of intersection of $\li{q_0}{q_1}$
 with $H$ (compare Figure~\ref{fig:hexagonsection}).

\begin{figure}[htpb]
\centering{\includegraphics[width=.9\textwidth]{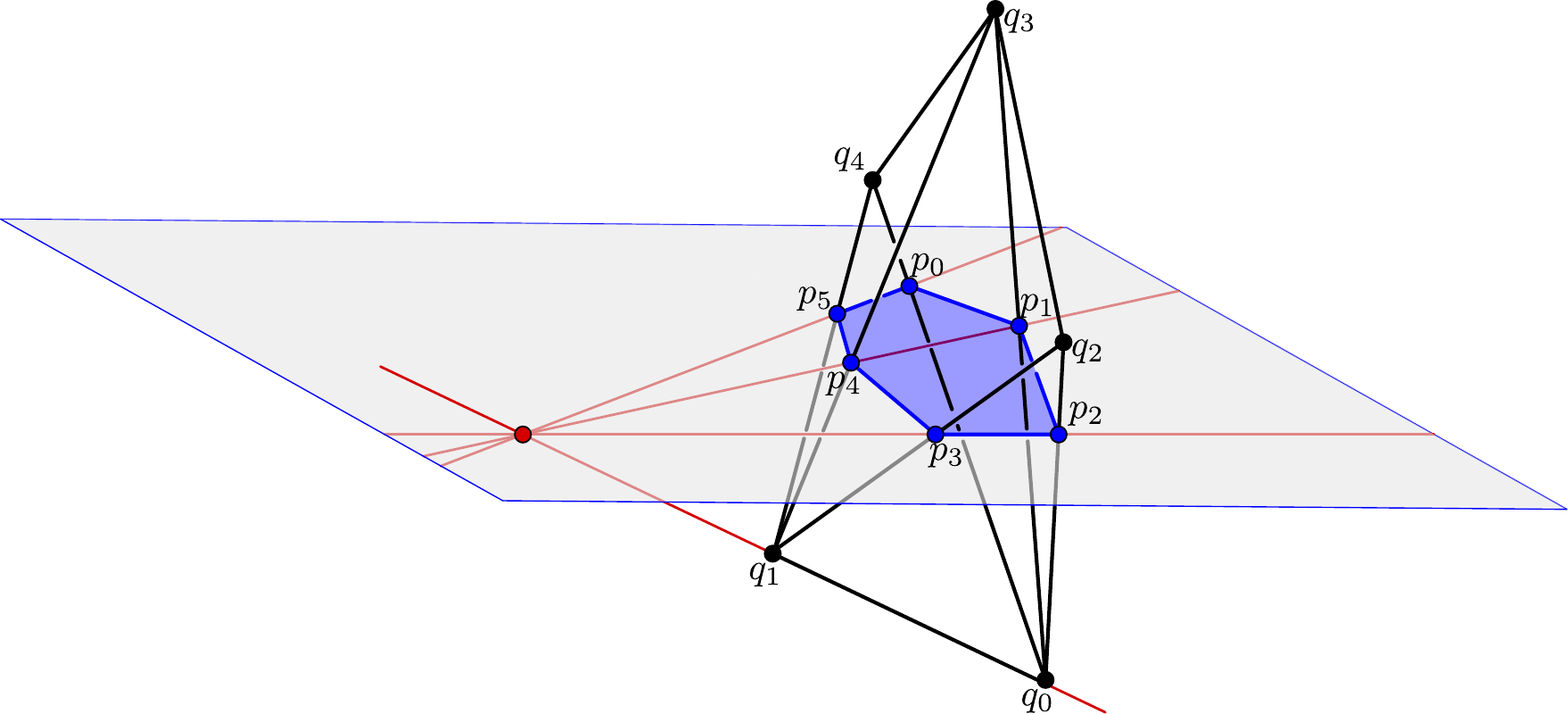}}

\caption{A hexagon as a section of a triangular bipyramid.}\label{fig:hexagonsection}
\end{figure}

 For the converse, we prove only the case when the point of intersection is finite (the 
 case with parallel lines is analogous). Then we can apply an affine transformation
 and assume that the coordinates of the hexagon are 
  \begin{align*}
  p_0&=(0,\alpha),& p_1&=(\beta x,\beta y),& p_2&=(\gamma,0)\\
  p_3&=(1,0),&p_{4}&=(x,y),& p_{5}&=(0,1);
  \end{align*}
 for some $x,y>0$ and $\alpha,\beta,\gamma>1$.

Now, let $K>\max(\alpha,\beta,\gamma)$ and consider the polytope $Q$ with vertices 
\begin{align*}
  q_0&=(0,0,-K),\qquad  q_1=(0,0,-1),\qquad  q_2=\frac{\big((K-1)\gamma,0,  K(\gamma-1)\big)}{K-\gamma},\\
  q_3&=\frac{\big(x(K-1)\beta,y(K-1)\beta, K(\beta-1)\big)}{K-\beta},
  \quad  q_4=\frac{\big(0,(K-1)\alpha,  K(\alpha-1)\big)}{K-\alpha}.
  \end{align*}
  If $H$ denotes the plane of vanishing third coordinate, then $q_0$ and $q_1$ lie below~$H$, while $q_2$, $q_3$ and $q_4$ lie above. The intersections $\li{q_i}{q_j}\cap H$ for $i\in\{0,1\}$ and $j\in\{2,3,4\}$ coincide with the vertices of $P\times\{0\}$. This proves that $P\times\{0\}=Q\cap H$, and hence that $\ic(P)=5$.
\end{proof}

\begin{remark}
Let $P$ be a regular hexagon, and let $Q$ be a polytope with $5$ vertices such that $Q\cap H=P$ for some plane~$H$. 
By the proof of Proposition~\ref{prop:ichexagon}, $Q$ is a triangular bipyramid and one of the two halfspaces
defined by $H$ contains only two vertices of $Q$: $q_0$ and $q_1$. Even more, since $P$ is regular,
the line $\li{q_0}{q_1}$ must be parallel to one of the edge directions of $P$ because, as we saw in the previous proof, the projective point $\li{q_0}{q_1}\cap H$ must coincide with the intersection of two opposite edges of $P$ (at infinity in this case).
This means that there are three different choices for the direction of the line $\li{q_0}{q_1}$; and shows that
the set of minimal extensions of~$P$ (which can be parametrized  by the vertex coordinates) is not connected, even if we consider 
its quotient space obtained after identifying extensions related by an admissible projective
transformation that fixes $P$ and those related by a relabeling of the vertices of $Q$. 
A similar behavior was already observed in~\cite{MondSmith2003} for the space
of nonnegative factorizations of nonnegative matrices of rank~$3$ and nonnegative rank~$3$.
\end{remark}

\begin{figure}[htpb]
\centering
{\includegraphics[width=.4\textwidth]{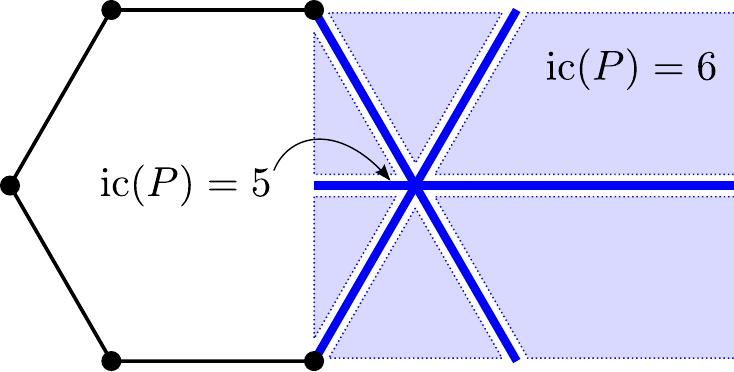}}
\caption{The intersection complexity of $P$ according to the position of the last vertex.}\label{fig:hexagonRS}
\end{figure}
\begin{remark}
Consider the set of all hexagons with $5$ fixed vertices. The position of the last vertex determines its intersection/extension complexity. This is depicted in Figure~\ref{fig:hexagonRS}. The hexagon fulfills the condition of Proposition~\ref{prop:ichexagon} if and only if the last point lies on any of the three dark lines. Hence, $\ic(P)=5$ if the last point lies on a dark line and $\ic(P)=6$ otherwise.
Actually, an analogous picture appears for any choice for the position of the 
initial $5$ points (the dark lines are always concurrent because of Pappus's Theorem).
In addition, the dark lines depend continuously on the coordinates of the first $5$ points.
This implies that, if we take two realizations that have the last point in two different 
$\ic(P)=6$ regions in Figure~\ref{fig:hexagonRS}, then we cannot continuously transform one into the other. Said otherwise, the realization space of the hexagon (as considered by Richter-Gebert in~\cite{RG}) restricted to those that have intersection complexity~$6$ is disconnected.
\end{remark}

\section{The complexity of heptagons}

In this section we prove our main result, Theorem~\ref{thm:icheptagon}, in two steps.
The easier part consists of showing that a special family of heptagons,
which we call standard heptagons, always have intersection complexity less than or equal to~$6$ (Proposition~\ref{prop:icstandardheptagon}).
The remainder of the section is devoted to proving the second step, Proposition~\ref{prop:noncrossingstandardization}: every heptagon is projectively equivalent to a standard heptagon.

\subsection{A standard heptagon}

Here, and throughout this section, $P$ denotes a heptagon and $\set{p_i}{i\in \ZZ/7\ZZ}$ is its set of vertices, cyclically clockwise labeled.

\begin{definition}
We say that $P$ is a \defn{standard heptagon} if
$p_0=(0,0)$, $p_3=(0,1)$ and $p_{-3}=(1,0)$; and
the lines $\li{p_1}{p_2}$ and $\li{p_{-1}}{ p_{-2}}$ are respectively parallel to the
lines $\li{p_0}{p_3}$ and $\li{p_0}{p_{-3}}$
(see~Figure~\ref{fig:heptagon}).
\end{definition}

We can easily prove that standard heptagons have intersection complexity at
most~$6$.

\begin{figure}[htbp]
\centering
\subcaptionbox{Coordinates of a standard
heptagon.\label{fig:heptagon}}[.45\textwidth]
{\raisebox{1cm}{\includegraphics[width=.4\textwidth]{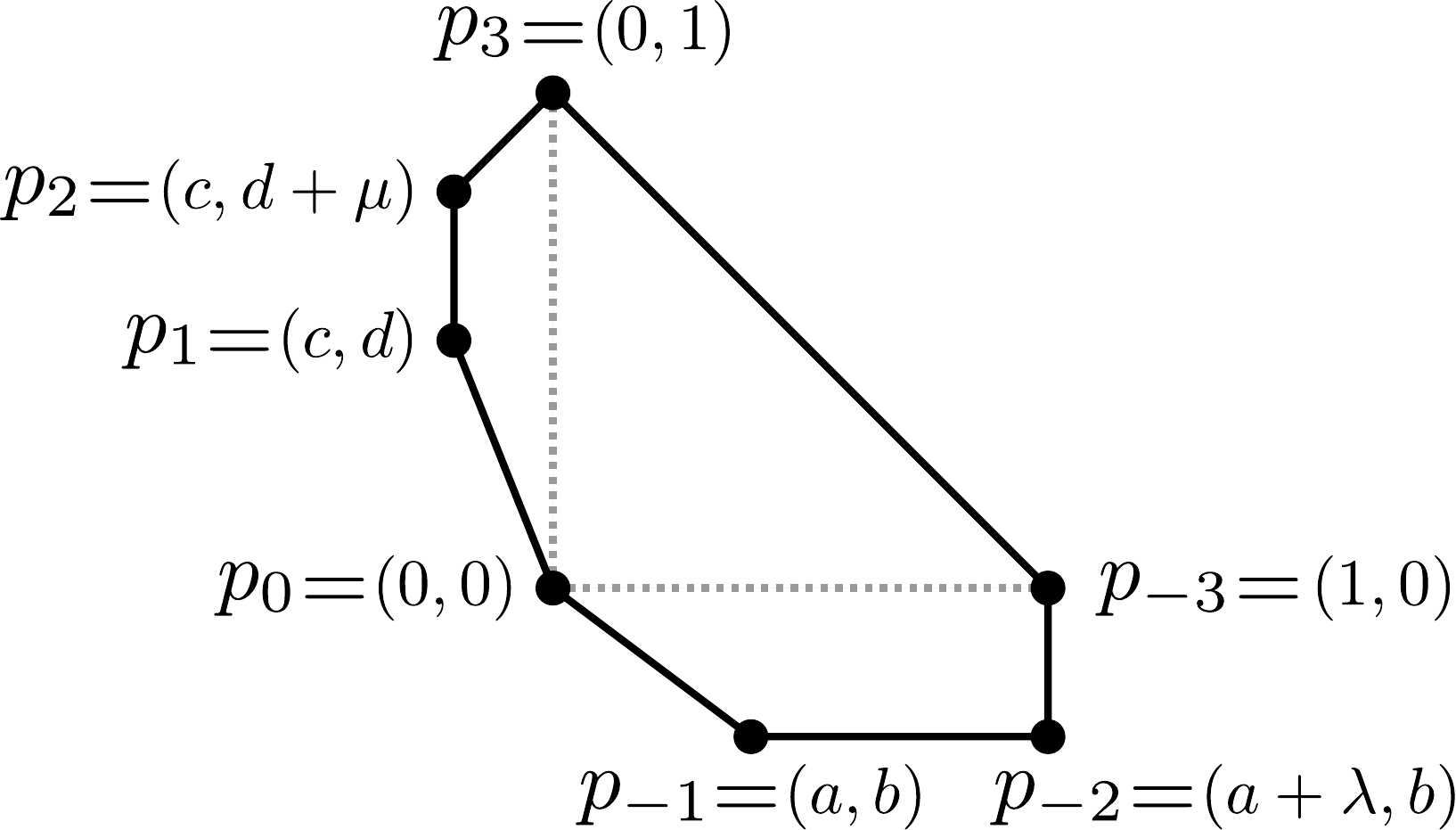}}}
\qquad
\subcaptionbox{The setup of
Proposition~\ref{prop:icstandardheptagon}.\label{fig:intersection}}[
.45\textwidth ]
{\includegraphics[width=.4\textwidth]{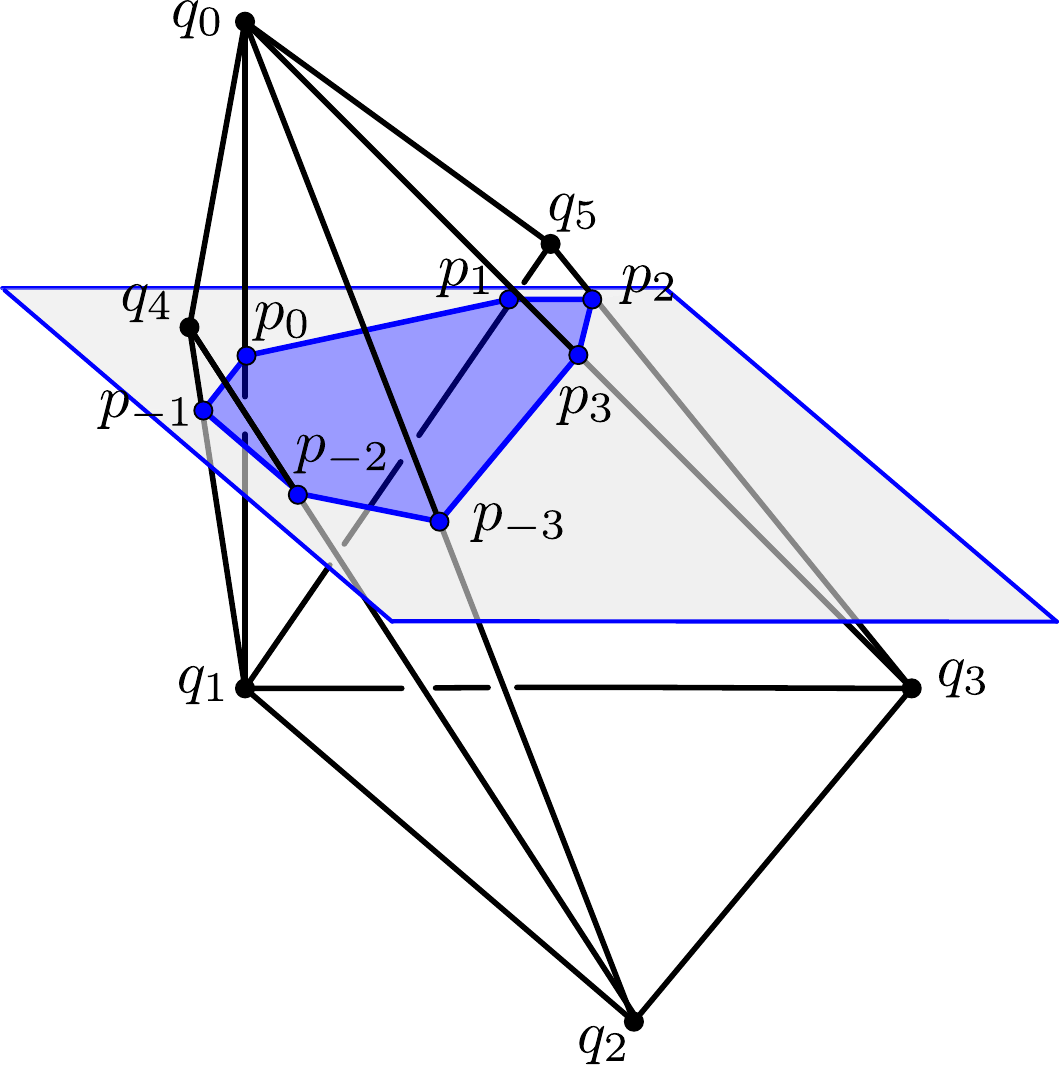}}
        \caption{Standard heptagons.}
\end{figure}

\begin{proposition}\label{prop:icstandardheptagon}
Every standard heptagon $P$ is a section of a $3$-polytope $Q$ with $6$ vertices.
In particular $\ic(P)\leq 6$.
\end{proposition}

\begin{proof}
For any standard heptagon~$P$, there are real numbers $b,c < 0 < a,d,\lambda,\mu$ 
such that the coordinates of the vertices of $P$ are
\begin{align*}
p_0&=(0,0),& p_1&=(c,d),& p_2&=(c,d+\mu),&p_3&=(0,1),\\
&&p_{-1}&=(a,b),& p_{-2}&=(a+\lambda,b), & p_{-3}&=(1,0). 
\end{align*}

\noindent Fix some $K>\max(\lambda -1,\mu -1)$ and consider the points 
\begin{align*}
q_0&:=(0,0,1),\,\, q_1:=(0,0,-K),\,\,q_2:=(1+K,0,-K),\,\,q_3:=(0,1+K,-K),\\ 
q_4&:=\frac{\big(a(1+K),b(1+K),
\lambda K\big)}{(1+K)-\lambda},\,\,
q_5:=\frac{\big((1+K)c,(1+K)d,\mu
K\big)}{(1+K)-\mu}.
\end{align*}

\noindent We claim that $P$ is the intersection of the $3$-polytope 
$Q:=\conv\{q_0,q_1,\dots,q_5\}$ with the plane $H:=\set{(x,y,z)\in\RR^3}{z=0}$:
\[
   Q\cap H=P\times \{0\}.
\]

Observe that every vertex of $Q\cap H$ corresponds to the intersection of~$H$ with
an edge of~$Q$ that has one endpoint on each side of the plane. Since 
$q_0$, $q_4$ and $q_5$ lie above $H$ and $q_1$, $q_2$ and $q_3$
lie below, the intersections of the relevant lines $\li{ q_i}{q_j}$ with~$H$
are (see Figure~\ref{fig:intersection}):
\begin{align*}
\li{q_0}{q_1}\cap H&=(0,0,0),& \li{q_0}{q_2}\cap H&=(1,0,0),& \li{q_0}{q_3}\cap H&=(0,1,0),\\
\li{q_4}{q_1}\cap H&=(a,b,0),&\li{q_4}{q_2}\cap H&=(a+\lambda,b,0),&\li{q_4}{q_3}\cap H&=(a,b+\lambda,0),\\
\li{q_5}{q_1}\cap H&=(c,d,0),&\li{q_5}{q_2}\cap H&=(c+\mu,d,0),&\li{q_5}{q_3}\cap H&=(c,d+\mu,0) .
\end{align*}

These are the vertices of $P\times\{0\}$ together with $(a,b+\lambda,0)$,
$(c+\mu,d,0)$, which proves that $Q\cap H\supseteq P\times \{0\}$.
To prove that indeed $Q\cap H= P\times \{0\}$, we need to see that 
both $(a,b+\lambda)$ and $(c+\mu,d)$ belong to $P$.

The convexity of $P$ implies the following conditions on the coordinates of its vertices by comparing, respectively, $p_{-1}$ with the lines $\li{p_0}{p_3}$ and $\li{p_0}{p_{-3}}$, $p_{-1}$ with $p_{-2}$, and $p_{-2}$ with the line $\li{p_3}{p_{-3}}$:
\begin{align*}
 a&>0, & -b&>0, & \lambda&>0, & 1-a-b-\lambda>0.
\end{align*}

Hence, the real numbers $\frac{1-a-b-\lambda}{1-b}$, $\frac{a}{1-b}$ and $\frac{\lambda}{1-b}$ are all greater than $0$. Since they add up to $1$, we can exhibit $(a,b+\lambda)$ as a convex combination
of $p_{-1}$, $p_{-2}$ and $p_3$:
\begin{equation*}
 \frac{1-a-b-\lambda}{1-b}\,p_{-1}+\frac{a}{1-b}\,p_{-2}+\frac{\lambda}{1-b}\,p_3=(a,b+\lambda).
\end{equation*}
This proves that $(a,b+\lambda)\in P$. That $(c+\mu,d)\in P$ is proved analogously.
\end{proof}

\subsection{Standardization lines of heptagons}

Our next goal is to show that every heptagon is projectively equivalent to a standard heptagon. For this, the key concept is that of a standardization line.

\begin{definition}
Consider a heptagon $P$, embedded in the projective space~$\PP^2$, whose vertices are cyclically labeled $\set{p_i}{i\in \ZZ/7\ZZ}$.
For $i\in \ZZ/7\ZZ$, and abbreviating $\li{i}{j}:=\li{p_i}{p_j}$, construct
\begin{align*}
p_i^+ &:= \li{i+ 1}{i+ 2}\vee\li{i}{i+ 3},&
p_i^- &:= \li{i- 1}{i- 2}\vee\li{i}{i- 3},&
\ell_i &:= p_i^+ \wedge p_i^-.
\end{align*}

We call the line $\ell_i$ the $i$th \defn{standardization line} of~$P$. If $\ell_i\cap P=\emptyset$, it is a \defn{non-crossing} standardization line.
\end{definition}
Figure~\ref{fig:standardizationlines} shows a heptagon and its 
standardization lines $\ell_0$ and $\ell_ {-3}$. Observe that $\ell_0$~is a non-crossing  standardization line, while $\ell_{-3}$ is not. 
\begin{figure}[htpb]
\includegraphics[width=.45\linewidth]{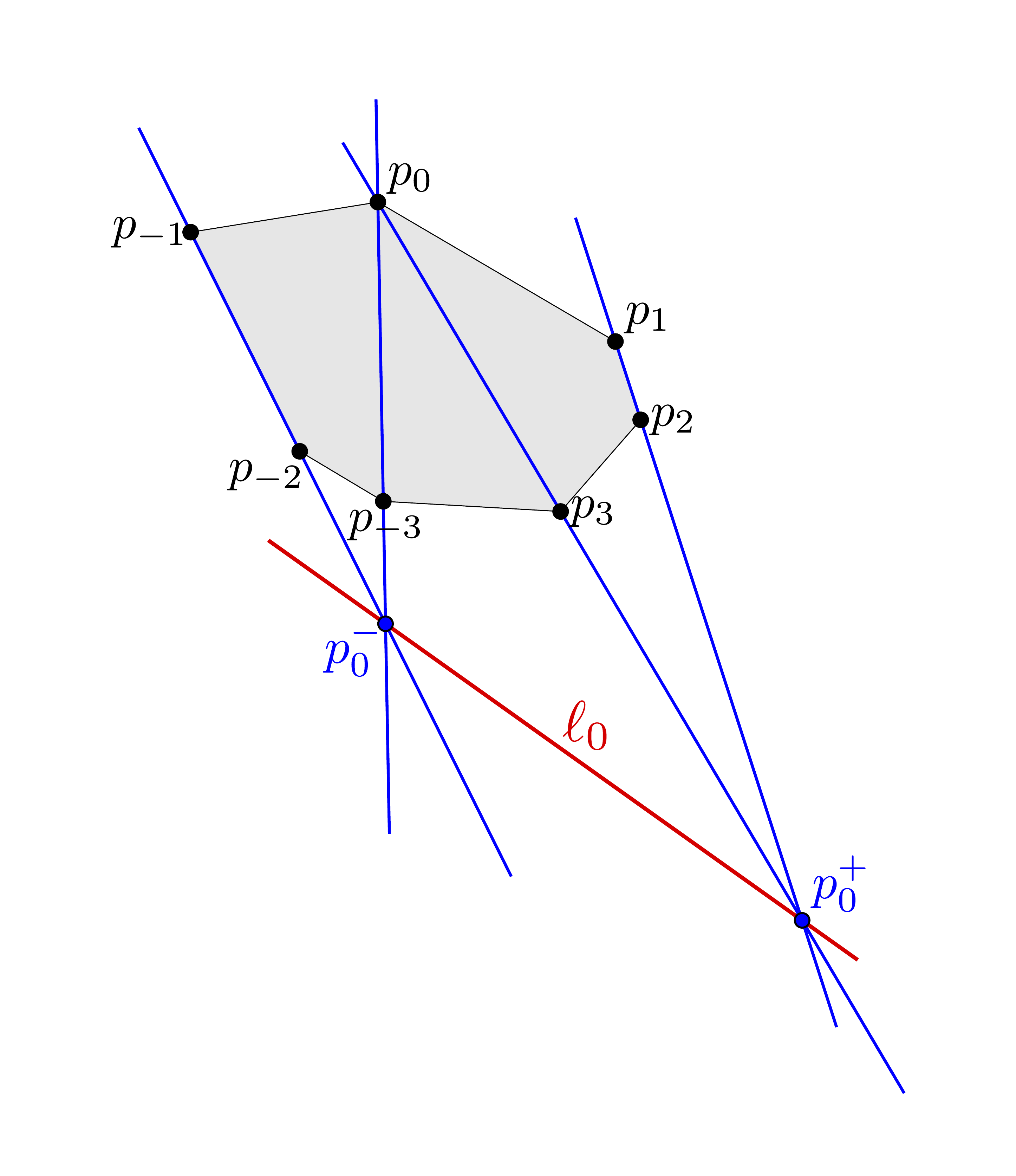}
\includegraphics[width=.45\linewidth]{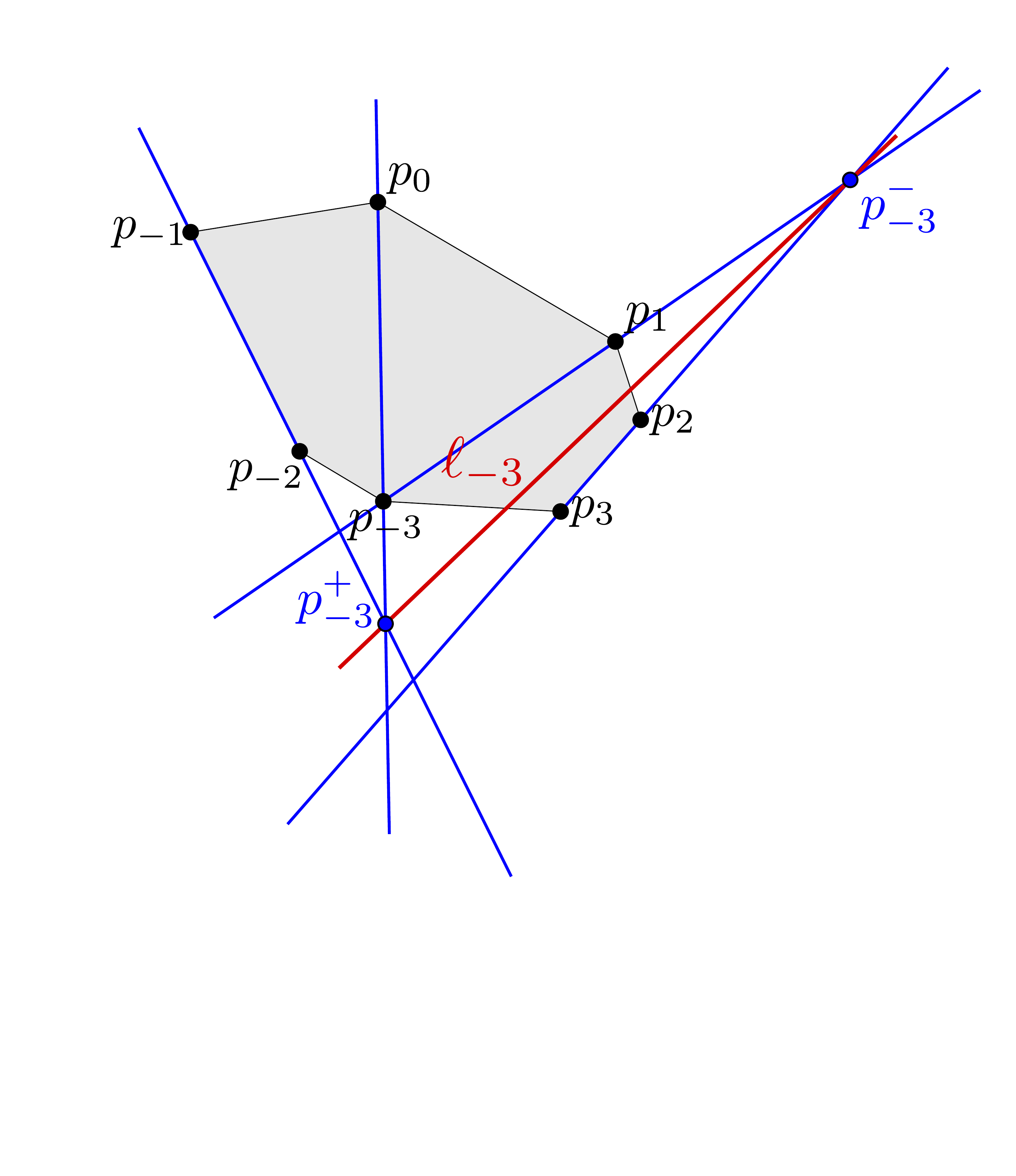}
\vspace{-.8cm}
\caption{The standardization lines $\ell_0$ and
$\ell_ {-3}$.}\label{fig:standardizationlines}
\end{figure}

\begin{lemma}\label{lem:converttostandard}
 A heptagon $P$ is projectively equivalent to a standard heptagon if  and only if it has at least one non-crossing standardization line.
\end{lemma}
\begin{proof}
The line at infinity of a standard heptagon must be one of its standardization
lines, which is obviously non-crossing. Conversely, the projective
transformation that sends a non-crossing standardization line of $P$ to
infinity, followed by a suitable affine transformation, maps $P$ onto a
standard heptagon.
\end{proof}

Hence, having a non-crossing standardization line characterizes standard
heptagons up to projective equivalence. Our next step is to show that every
heptagon has a non-crossing standardization line
(Proposition~\ref{prop:noncrossingstandardization}). But to prove this, we still
need to introduce a couple of concepts.

Observe that $\ell_i$ cannot cross any
of the lines $\li{p_{i+1}}{p_{i+2}}$, $\li{p_{i+0}}{p_{i+3}}$,
$\li{p_{i-1}}{p_{i-2}}$ and $\li{p_{i+0}}{p_{i-3}}$ in the interior of $P$,
since by construction their intersection point is $p_i^\pm$, which lies outside~$P$
(compare Figure~\ref{fig:standardizationlines}).
In particular, if $\ell_i$ intersects~$P$, either it separates $p_{i+1}$ and~$p_{i+2}$ from the remaining vertices of~$P$, or it separates $p_{i-1}$ and~$p_{i-2}$.

\begin{definition}
If the standardization line $\ell_i$ separates $p_{i+1}$ and $p_{i+2}$ from the remaining vertices of $P$, we say that it is \defn{$+$-crossing}; if it separates $p_{i-1}$~and~$p_{i-2}$ it is \defn{$-$-crossing}.
In the example of Figure~\ref{fig:standardizationlines}, $\ell_{-3}$ is $-$-crossing.

\end{definition}

\begin{definition}
The lines $\li{p_{i}}{p_{i+3}}$ and $\li{p_{i+1}}{p_{i+2}}$ partition the projective plane~$\PP^2$ into two disjoint angular sectors (cf.
Figure~\ref{fig:sector}). One of them contains the points 
$p_{i-1}$, $p_{i-2}$ and $p_{i-3}$, while the interior of the other is empty of vertices of~$P$. We denote this empty sector~\defn{$S_i^+$}.
Similarly,  \defn{$S_i^-$}~is the sector formed by
$\li{p_{i}}{p_{i-3}}$ and $\li{p_{i-1}}{p_{i-2}}$ that contains no vertices of~$P$.
\end{definition}

\begin{figure}
\centering
\subcaptionbox{The sector $S_0^+$ (shaded). The point $p_0^-$ lies in
$S_0^+$  if and only if $\ell_{0}$ is
$+$-crossing.\label{fig:sector}}[.45\textwidth]
{\includegraphics[width=.4\textwidth]{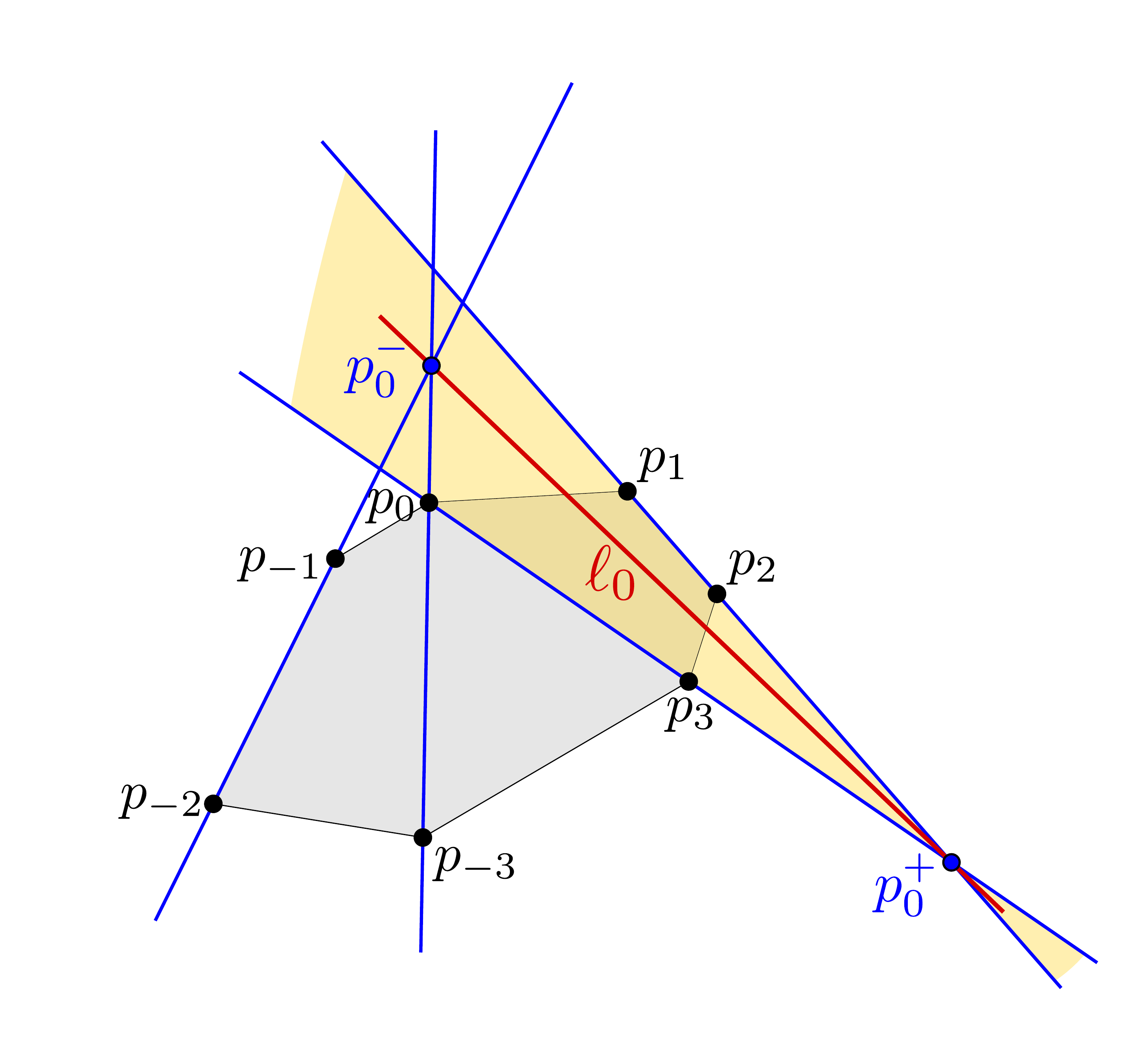}}\qquad
\subcaptionbox{The point $p_0^-$ cannot lie in both shaded sectors
simultaneously.\label{fig:sectorcompatibility}}[.45\textwidth]
{\includegraphics[width=.4\textwidth]{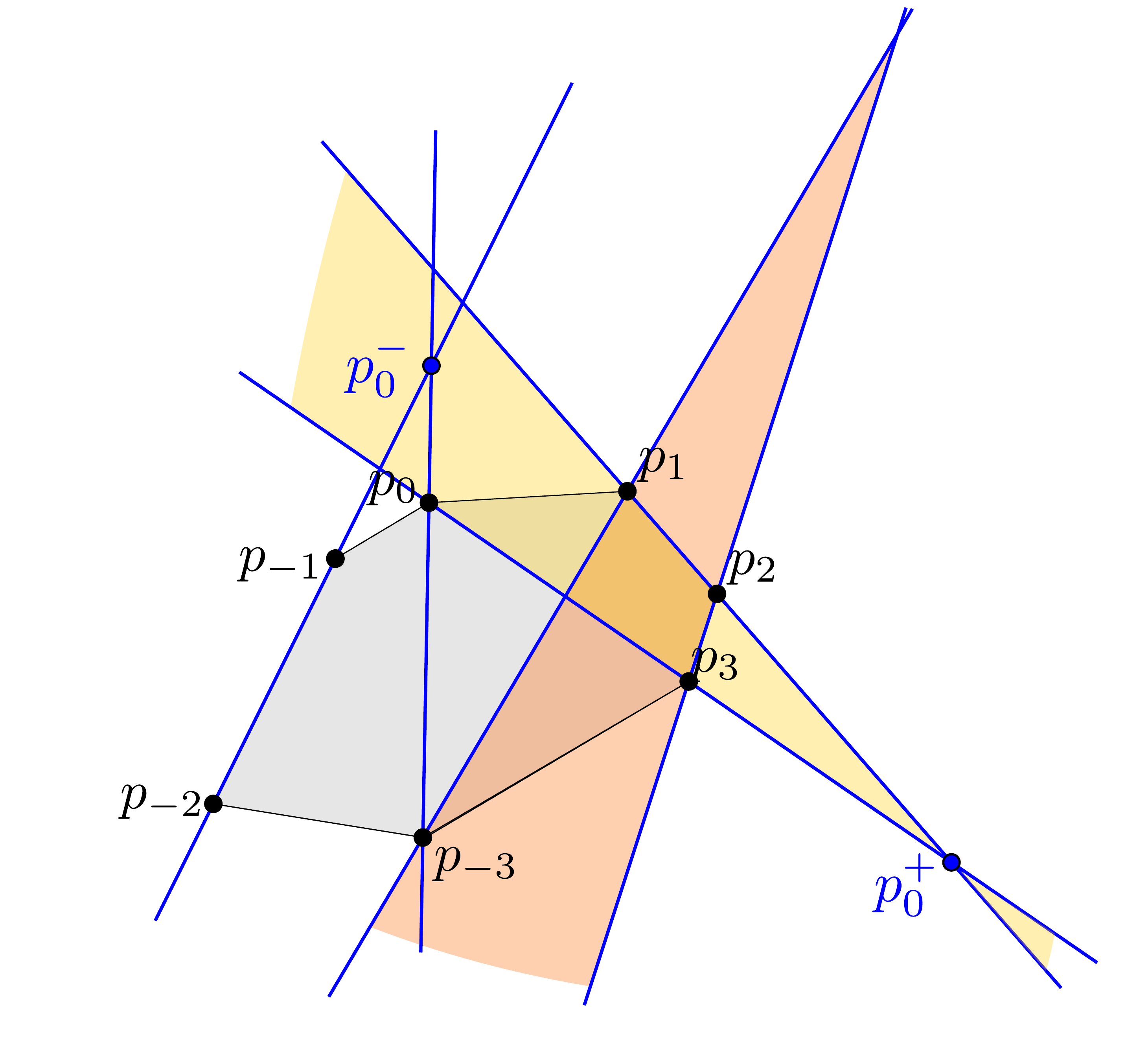}}
        \caption{The relevant angular sectors.}\label{fig:sectors}
\end{figure}

These sectors allow us to characterize $\pm$-crossing standardization lines.
\begin{lemma}\label{lem:crossingcharacterization}
 The standardization line $\ell_i$ is $+$-crossing if and only if  $p_i^-\in S_i^+$.
 Analogously, $\ell_i$ is $-$-crossing if and only if $p_i^+\in S_i^-$.
 \end{lemma}
\begin{proof}
The line $\ell_i$ is $+$-crossing when it separates $p_i$ and $p_{i+1}$ from the rest of~$P$; this happens if and only if $\ell_i\subset S_i^+$. Since $\li{p_{i-1}}{p_{i-2}}\cap\ell_i = p_i^-$, this is equivalent to $p_i^-\in S_i ^+$.
The case of $-$-crossing follows analogously.
\end{proof}

 With this characterization, we can easily prove the following compatibility condition.
\begin{lemma}\label{lem:sectorcompatibility}
 If $\ell_i$ is $+$-crossing, then $\ell_{i-3}$ cannot be $-$-crossing. Analogously, if $\ell_i$ is $-$-crossing, then $\ell_{i+3}$ cannot be $+$-crossing.
\end{lemma}
\begin{proof}
 Both statements are equivalent by symmetry. We assume that $\ell_i$ is $+$-crossing and $\ell_{i-3}$ is $-$-crossing to reach a contradiction.

 Observe that $p_i^-=p_{i-3}^+$ by definition. By Lemma~\ref{lem:crossingcharacterization}, $p_i^-$ must lie both in the sector formed by $\li{p_{i}}{p_{i+3}}$ and $\li{p_{i+1}}{p_{i+2}}$ and in the sector formed by $\li{p_{i-3}}{p_{i+1}}$ and $\li{p_{i+3}}{p_{i+2}}$. However, the intersection of these two sectors lies in the interior of the polygon (cf. Figure~\ref{fig:sectorcompatibility}), while $p_i^-$ lies outside.
\end{proof}

\begin{corollary}
  \label{cor:allcrossing}
  If all the standardization lines $\ell_i$ intersect $P$, they are either all $+$-crossing or all $-$-crossing.
\end{corollary}

\subsection{Every heptagon has a non-crossing standardization line}

We are finally ready to present and prove
Proposition~\ref{prop:noncrossingstandardization}.
In essence, we prove that the combinatorics of the pseudo-line arrangement in
Figure~\ref{fig:nonrealizable} are not realizable by an arrangement of 
straight lines in the projective plane. Here the ``combinatorics'' refers to the order 
of the intersection points in each projective pseudo-line.
However, any heptagon that had only $+$-crossing standardization lines would provide such
a realization (compare the characterization of Lemma~\ref{lem:crossingcharacterization}).

\begin{figure}[htpb]
 \centering
 \includegraphics[width=.55\linewidth]{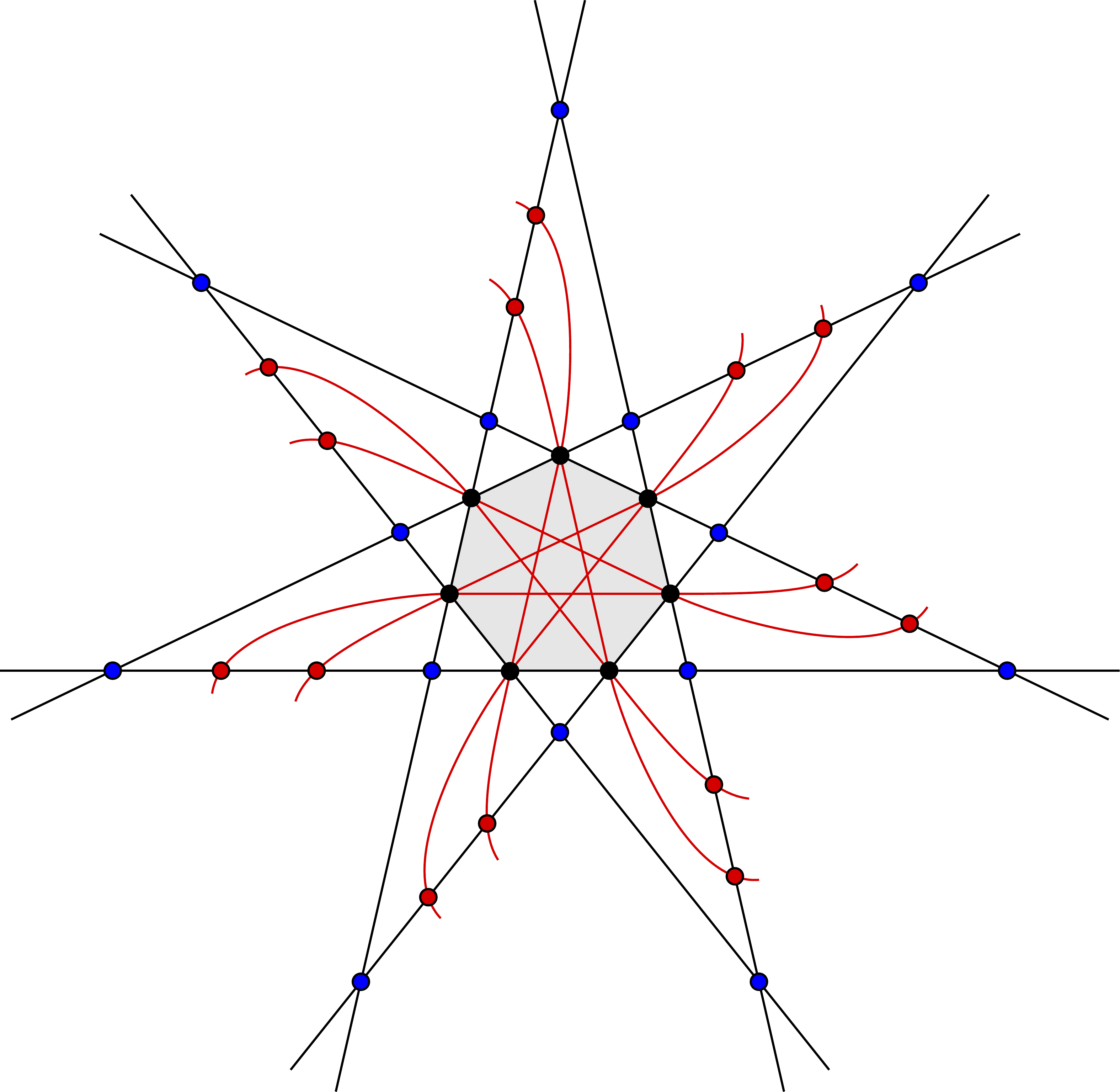}
 \caption{A non-stretchable pseudo-line arrangement.}\label{fig:nonrealizable}
\end{figure}

For the proof, we will need the formula \begin{equation}\label{eq:quadprod}  ({a \times b} ){\times} ({c}\times {d}) = [{a,\, b,\, d}]\,  c - [{a,\, b,\, c}] \, d ,
  \end{equation}
  where $[{a,\, b,\, c}]= \det(a,b,c)$ is the $3\times 3$-determinant formed by
the homogeneous coordinates of the corresponding points. That is, 
$[{p_x,\ p_y, \ p_z}]=\pm 2 \Vol\left(\conv\{{p_x,\ p_y, \ p_z}\}\right)$.
Observe that, since the vertices of the heptagon are labeled clockwise and are in convex position, $ [{p_x,\ p_y, \ p_z}]>0$ whenever $z$ lies in the interval $(x,y)$.
 To simplify the notation, in what follows we abbreviate $[{p_{i+x},\ p_{i+y}, \ p_{i+z}}]$ as $[x,y,z]_i$, for any $x,y,z,i\in \ZZ/7\ZZ$.

\begin{lemma}\label{lem:alg_characterization}
 With the notation from above, the standardization line $\ell_i$ is $+$-crossing
 if and only if 
  \begin{align}\label{eq:bdycondition1}
   [p_{i+2},p_{i+1}, p_i^-]= [- 1,-2,-3]_{i}[2, 1, 0]_{i} - [- 1,-2,0]_i[2, 1,-3]_{i}\geq 0;
  \end{align}
and is $-$-crossing if and only if 
\begin{align}\label{eq:bdycondition2}
[p_{i-2},p_{i-1},p_i^+]= [1,2,3]_i[-2,-1, 0]_{i} - [1,2,0]_{i}[-2,-1, 3]_{i} \geq 0.
  \end{align}
\end{lemma}
\begin{proof}
 Using \eqref{eq:quadprod}, the coordinates of the standardization point
$p_i^-$ are given by
  \begin{align*}
    p_i^-
    &= (p_{i- 1}\wedge p_{i- 2})\vee(p_{i}\wedge p_{i- 3})
    =(p_{i- 1}\times p_{i- 2})\times(p_{i}\times p_{i- 3})\\
    &\stackrel{\eqref{eq:quadprod}}{=}
    \phantom{-}[- 1,-2, -3]_ i\  p_{i} - [- 1,-2,0]_i\  p_{i-3}.
  \end{align*}

  Observe that $[p_{i+3},p_i,p_i^-]<0$ since
  \begin{align*}
    [p_{i+3},p_i,p_i^-]&=
    [- 1,-2, -3]_ i[3,0, 0]_ i - [- 1,-2,0]_i[3,0,-3]_ i\\
     &=\phantom{[- 1,-2, -3]_ i[3,0, 0]_ i} - [- 1,-2,0]_i[3,0, -3]_ i
     \ < \ 0,
  \end{align*}
  because $[3,0,0]_i=0$ and $[- 1,-2,0]_i > 0$, $[3,0, -3]_ i > 0$ by convexity. %

  Therefore, in view of Lemma~\ref{lem:crossingcharacterization}, requiring  $\ell_i$ to be $+$-crossing reduces to the equation $[p_{i+2},p_{i+1}, p_i^-]\geq 0$, since otherwise $ p_i^-$ would not lie in the desired sector. This expression can be reformulated as~\eqref{eq:bdycondition1}. The proof of~\eqref{eq:bdycondition2} is analogous.
\end{proof}

\begin{proposition}\label{prop:noncrossingstandardization}
 Every heptagon has at least one non-crossing standardization line.
\end{proposition}

 \begin{proof}
  We want to prove that $P$ has at least one non-crossing standardization line.
  By Corollary~\ref{cor:allcrossing} (and symmetry), it is enough to prove that
it is impossible for all $\ell_i$ to be $+$-crossing. We will assume this to be
the case and reach a contradiction.

If $\ell_i$ is $+$-crossing for all $0\leq i\leq 6$, then by Lemma~\ref{lem:alg_characterization}, the coordinates of the vertices of $P$ fulfill 
\eqref{eq:bdycondition1} for all $i\in\ZZ/7\ZZ$. Moreover, if $\ell_i$ is $+$-crossing then it cannot be $-$-crossing. Therefore, again by Lemma~\ref{lem:alg_characterization}, one can see that the coordinates of the vertices of $P$ fulfill
  \begin{align}\label{eq:bdycondition3}
   [2,1,0]_{i}[-1,-2, 3]_{i} - [2,1,3]_i[-1,-2, 0]_{i}> 0,
  \end{align}
for all $i\in\ZZ/7\ZZ$.

  Therefore, if all the $\ell_i$ are $+$-crossing, the addition of the left hand sides of \eqref{eq:bdycondition1} and \eqref{eq:bdycondition3} for all $i\in\ZZ/7\ZZ$ should be positive. 
  
  With the abbreviations
\begin{align*}
 A_i&:=[- 1,-2,-3]_{i},& B_i&:=[2, 1, 0]_{i},& C_i&:=[- 1,-2,0]_i,\\
 D_i&:=[2, 1,-3]_{i},&  E_i&:=[2,1,0]_{i},& F_i&:=[-1,-2, 3]_{i},\\
 G_i &:=[2,1,3]_i,& H_i&:=[-1,-2, 0]_{i};
\end{align*}
 this can be expressed as
\begin{equation}\label{eq:globalcondition}
 \sum_{i\in \ZZ/7\ZZ}A_iB_i-C_iD_i+E_iF_i-G_iH_i>0 .
\end{equation}
However, it turns out that for every heptagon the equation 
\begin{equation}\label{eq:heptagonidentity}
 \sum_{i\in \ZZ/7\ZZ}A_iB_i-C_iD_i+E_iF_i-G_iH_i=0
\end{equation}
holds by the upcoming Lemma~\ref{lem:invariant}. This contradiction concludes the proof 
 that every heptagon has at least one standardization line.
\end{proof}

\begin{lemma}\label{lem:invariant}
Let $A$ be a configuration of $7$ points in $\EE^2\subset\PP^2$ labeled
$\set{a_i}{i\in \ZZ/7\ZZ}$. Denote the determinant $[{a_{i+x},\ a_{i+y}, \
a_{i+z}}]$ as $[x,y,z]_i$, for any $x,y,z,i\in \ZZ/7\ZZ$. Finally, let 
\begin{align*}
 A_i&:=[- 1,-2,-3]_{i},& B_i&:=[2, 1, 0]_{i},& C_i&:=[- 1,-2,0]_i,\\
 D_i&:=[2, 1,-3]_{i},&  E_i&:=[2,1,0]_{i},& F_i&:=[-1,-2, 3]_{i},\\
 G_i &:=[2,1,3]_i,& H_i&:=[-1,-2, 0]_{i};
\end{align*}

Then,
\begin{equation}\tag{\ref{eq:heptagonidentity}}
 \sum_{i\in \ZZ/7\ZZ}A_iB_i-C_iD_i+E_iF_i-G_iH_i=0.
\end{equation}
\end{lemma}
\begin{proof}
Although~\eqref{eq:heptagonidentity} can be checked purely algebraically, we provide a geometric interpretation. Observe that 
$[x,y,z]_i=\pm2\Vol(a_{i+x},a_{i+y},a_{i+z})$, 
which implies that the identity in~\eqref{eq:heptagonidentity} can be proved in terms of
(signed) areas of certain triangles spanned by $A$. Figure~\ref{fig:invariant} depicts some of these triangles when the points are
in convex position.

\begin{figure}[htpb]
 \centering
\begin{tabular}{cccc}
\includegraphics[width=.18\linewidth]{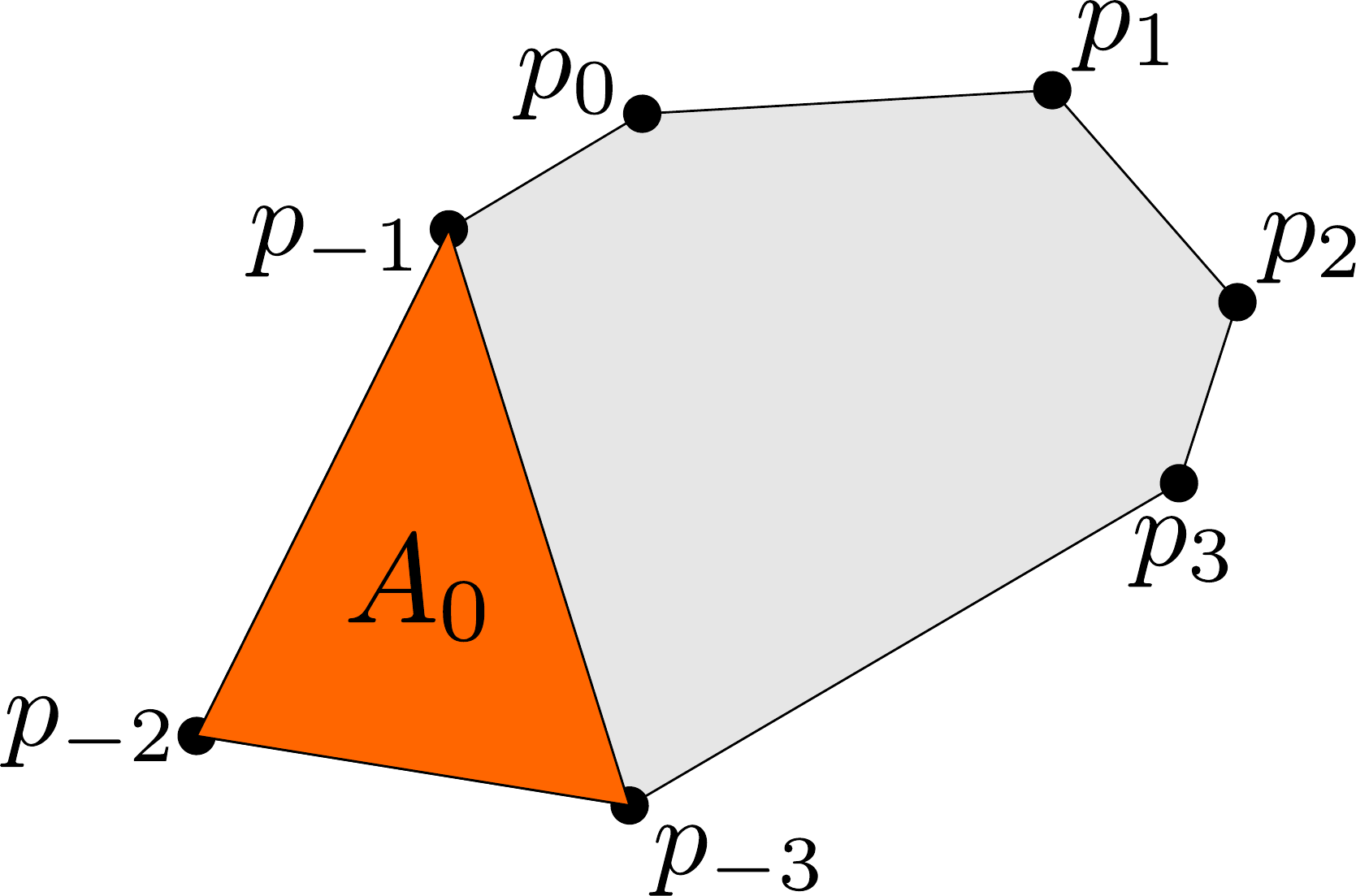}& 
\includegraphics[width=.18\linewidth]{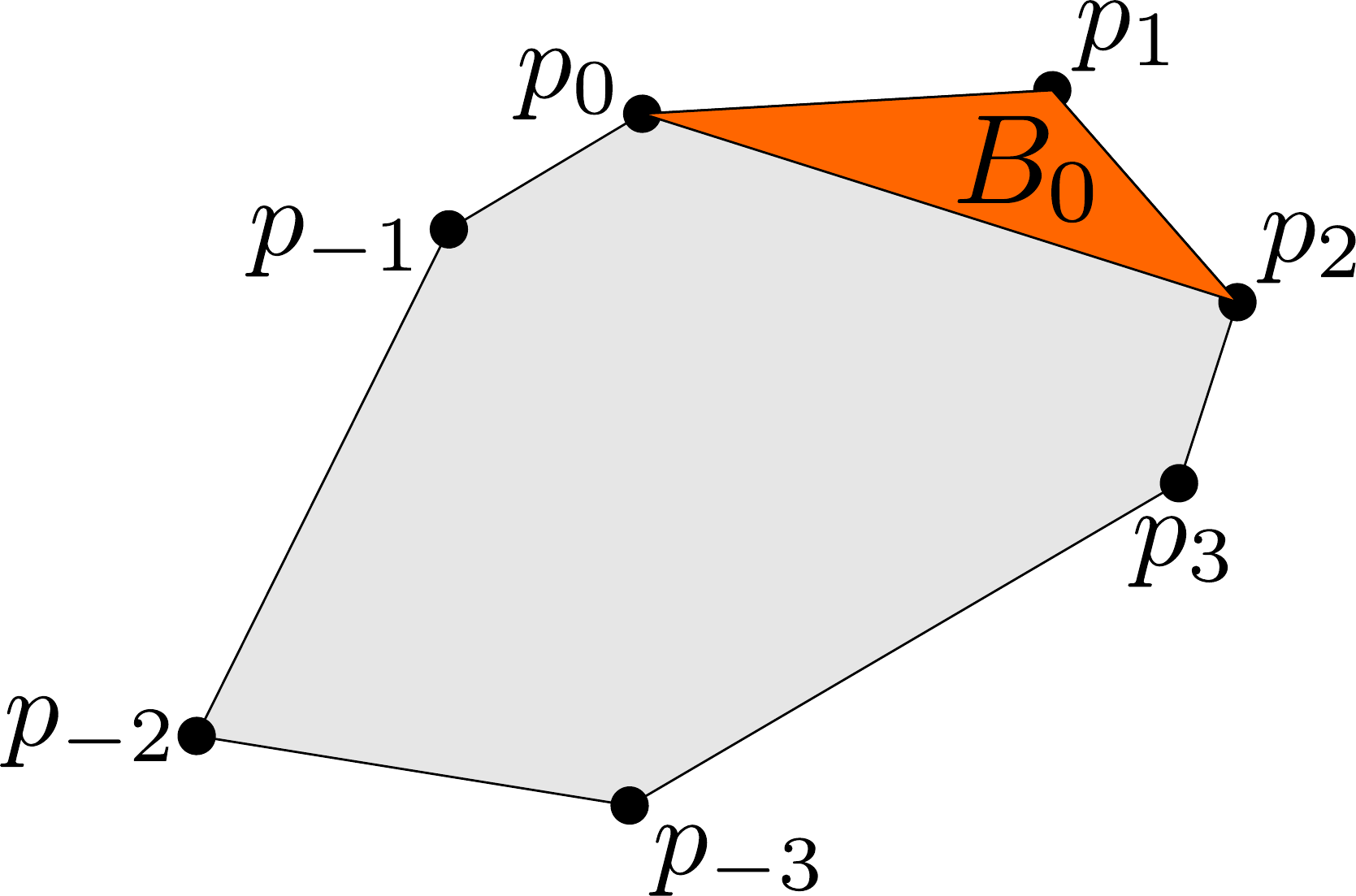}&
\includegraphics[width=.18\linewidth]{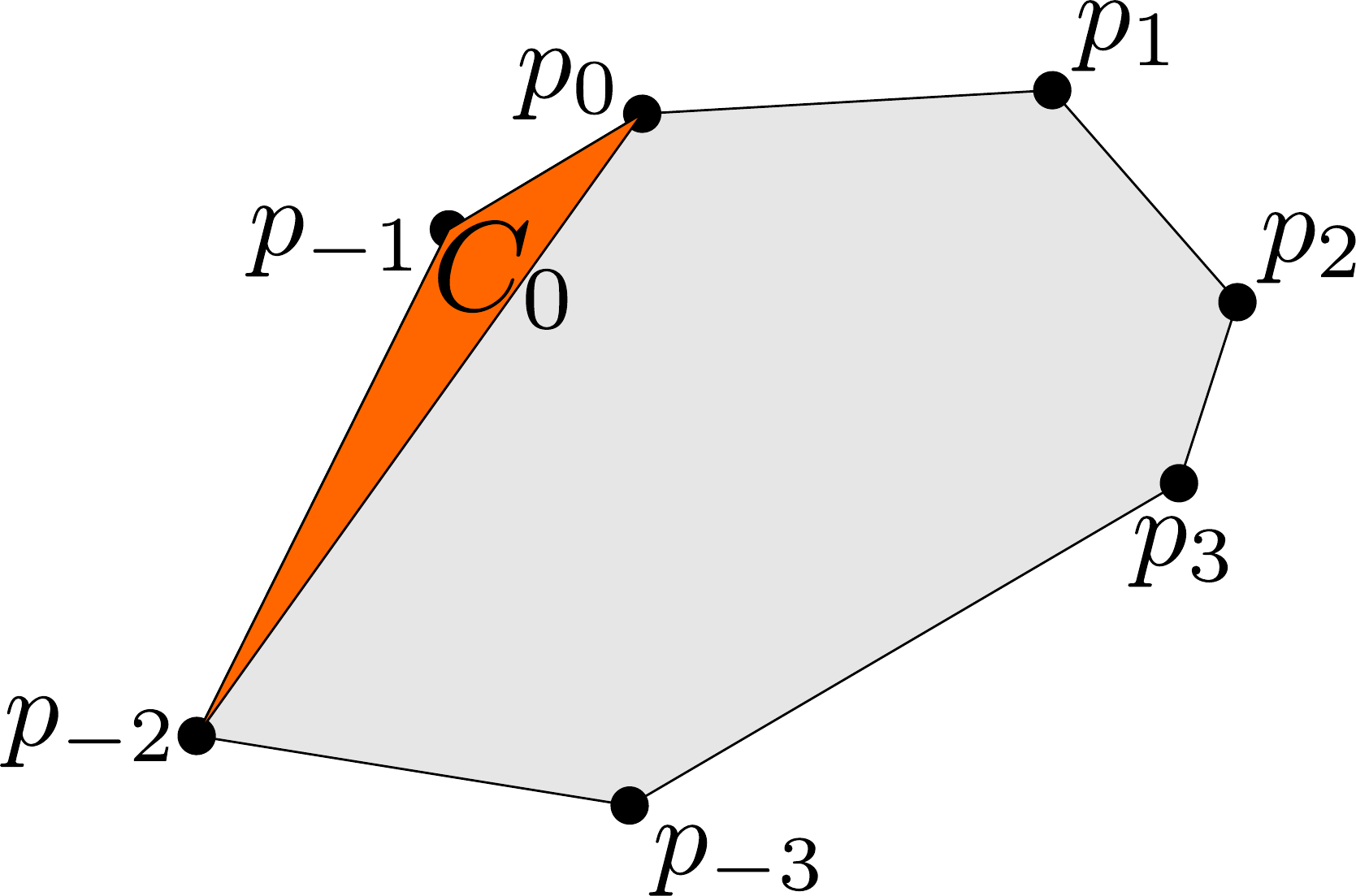}&
\includegraphics[width=.18\linewidth]{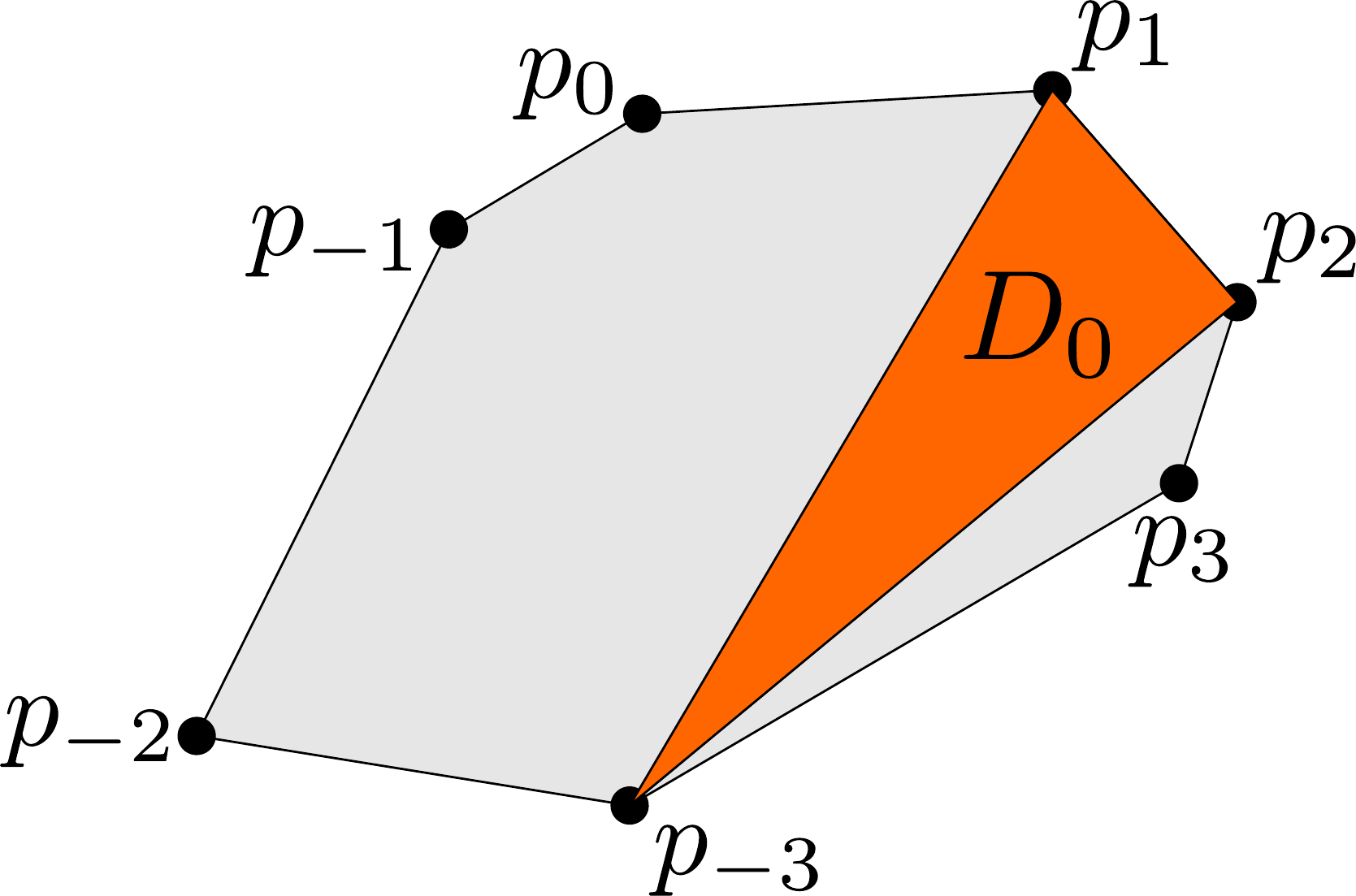}\\
$A_0$&$B_0$&$C_0$&$D_0$\\
\includegraphics[width=.18\linewidth]{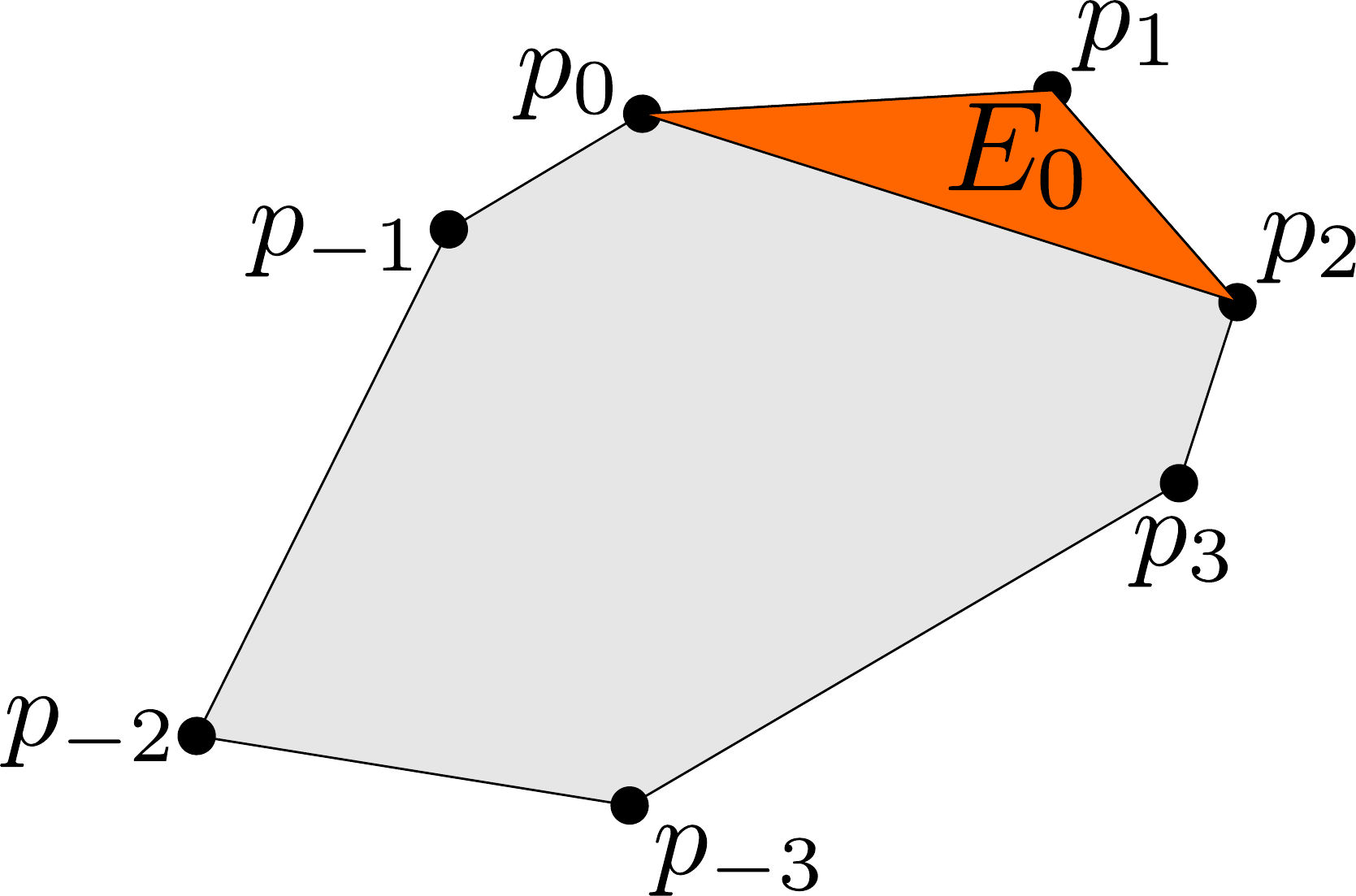}& 
\includegraphics[width=.18\linewidth]{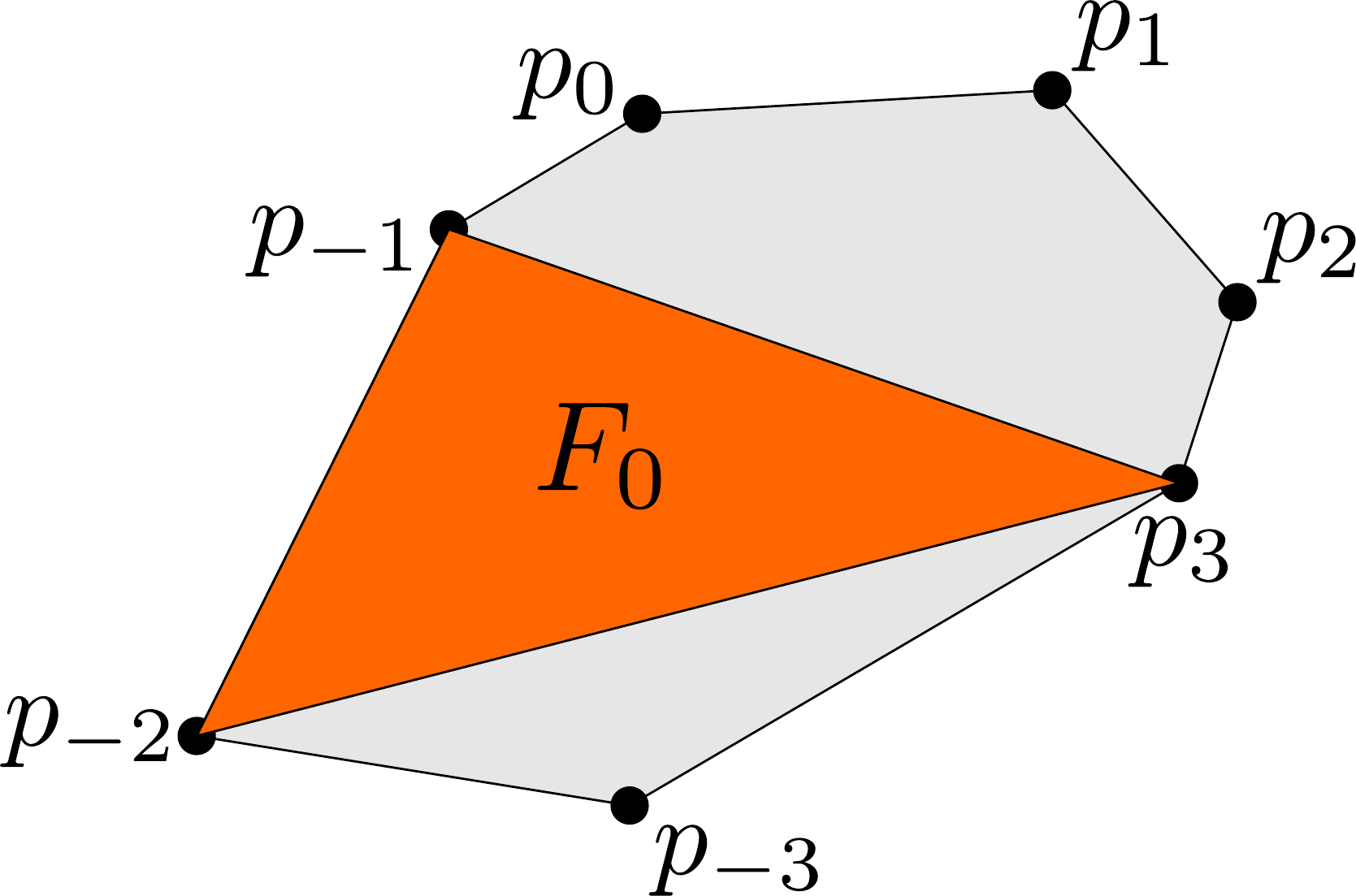}&
\includegraphics[width=.18\linewidth]{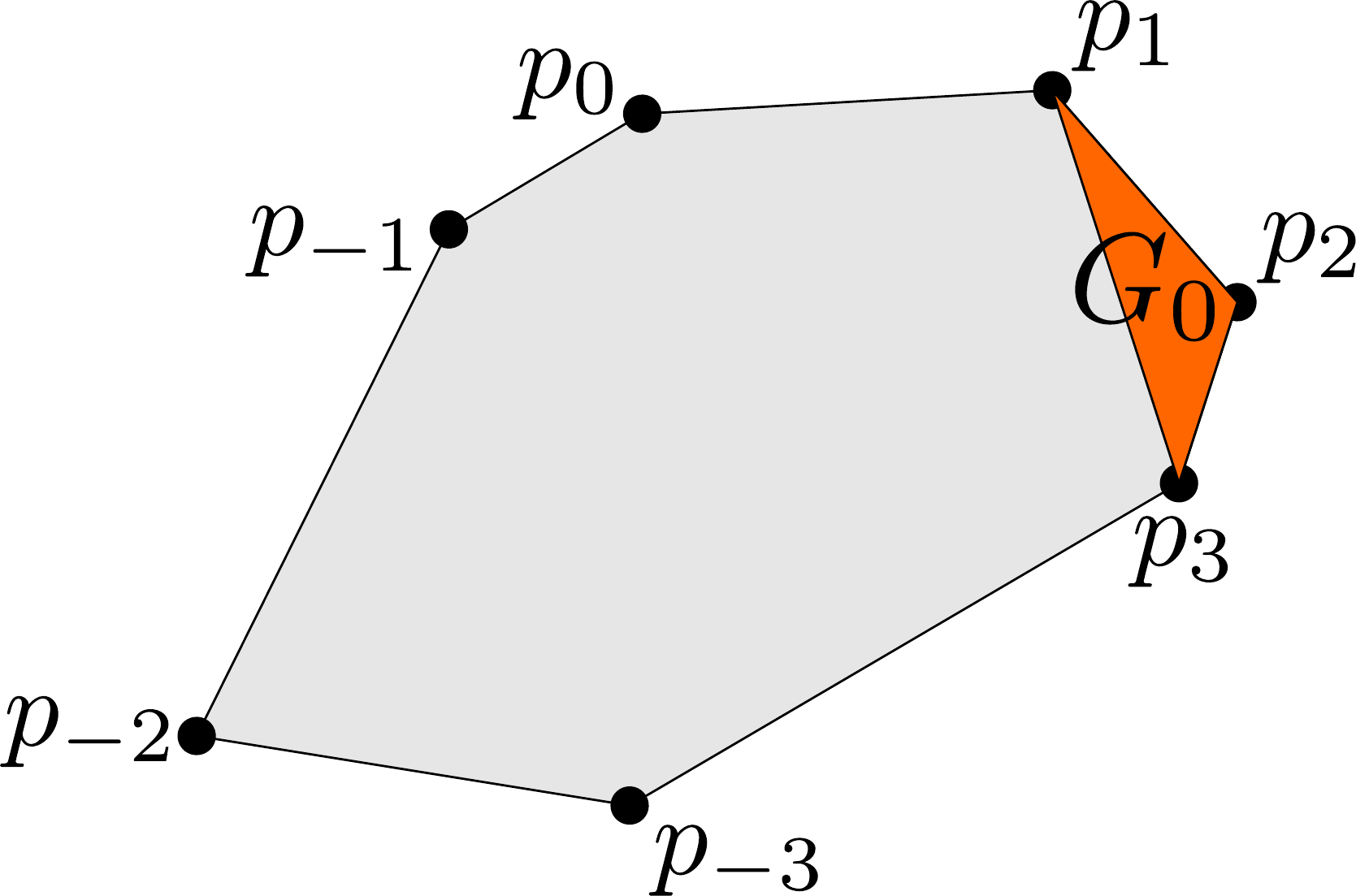}&
\includegraphics[width=.18\linewidth]{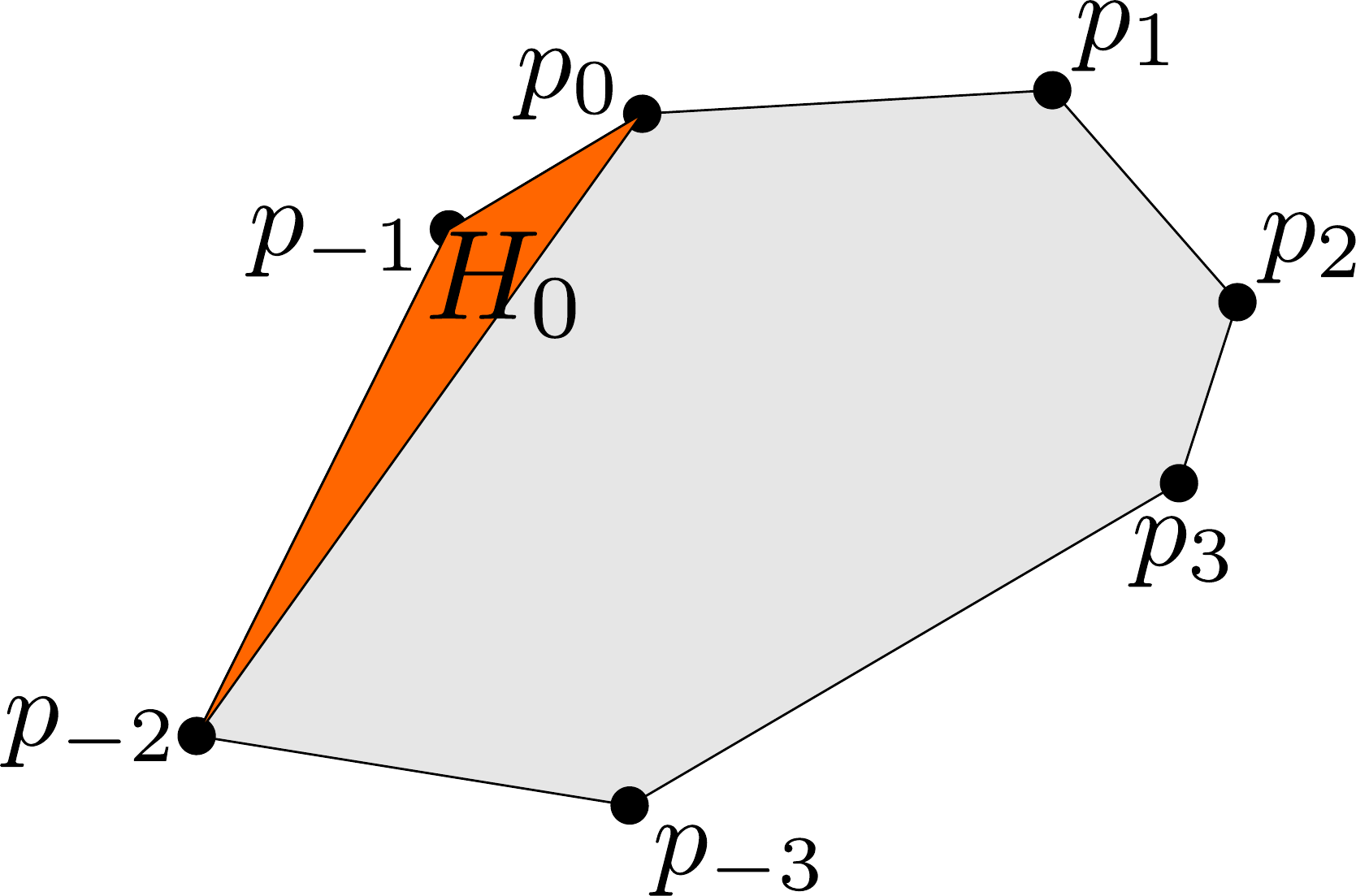}\\
$E_0$&$F_0$&$G_0$&$H_0$
\end{tabular}
\caption{The determinants involved in Lemma~\ref{lem:invariant}.}\label{fig:invariant}
\end{figure} 
 
To see \eqref{eq:heptagonidentity}, we show the stronger result that both
\begin{align}
\sum_{i\in \ZZ/7\ZZ}A_iB_i& =\sum_{i\in \ZZ/7\ZZ}G_iH_i \qquad \text{ and }\label{eq:identity1}\\\sum_{i\in \ZZ/7\ZZ}E_iF_i&=\sum_{i\in \ZZ/7\ZZ}C_iD_i.\label{eq:identity2}
\end{align}

Indeed, the identity \eqref{eq:identity1} is easy, because
$A_iB_i=G_{i+3}H_{i+3}$ for all $i\in \ZZ/7\ZZ$ since $A_i=G_{i+3}$ and $B_i=H_{i+3}$. 
 
Moreover, it is straightforward to check that
\begin{align}
C_i&=E_{i-2},\label{eq:Ci}
\end{align}
and it is also not hard to see that
\begin{equation}
 D_i+C_{i-3}=F_{i-2}+E_{(i+3)-2}.\label{eq:square}
 \end{equation}
Finally, we subtract the right-hand side of \eqref{eq:identity2} from the left hand side:
 \begin{align*}
 & \sum_{i\in \ZZ/7\ZZ}E_iF_i -\sum_{i\in \ZZ/7\ZZ}C_iD_i
  \ = \ 
  \sum_{i\in \ZZ/7\ZZ}E_{i-2}F_{i-2}-\sum_{i\in \ZZ/7\ZZ}C_iD_i
  \\ &\stackrel{\phantom{\eqref{eq:Ci}}}{=}
  \sum_{i\in \ZZ/7\ZZ}E_{i-2}(F_{i-2}+E_{i+1}-E_{i+1})
  -\sum_{i\in \ZZ/7\ZZ}C_i(D_i+C_{i-3}-C_{i-3})
  \\ &
  \stackrel{\eqref{eq:Ci}}{=}
  \sum_{i\in\ZZ/7\ZZ}C_i
  \underbrace{\big(F_{i-2}+E_{i+1}-D_i-C_{i-3}\big)}_{{}=0 \text{ by }\eqref{eq:square}}
  -\sum_{i\in \ZZ/7\ZZ}E_{i-2}E_{i+1}
  +\sum_{i\in \ZZ/7\ZZ}C_iC_{i-3}
  \\ &\stackrel{\eqref{eq:Ci}}{=}
  -\sum_{i\in \ZZ/7\ZZ}C_{i}C_{i+3} 
  +\sum_{i\in \ZZ/7\ZZ} C_iC_{i-3} \ = \ 0,
 \end{align*}
and this concludes our proof of~\eqref{eq:identity2}.
\end{proof}

\begin{observation}
In contrast to Proposition~\ref{prop:noncrossingstandardization}, there exist 
heptagons with $6$~crossing standardization lines.
For example, the convex hull of
\begin{align*}
p_0&=(\tfrac{7}{5},\tfrac{1}{2}),&p_1&=(\tfrac{6}{5},\tfrac{1}{10}), &p_2&=(1,0), &p_3&=(0,0), \\p_{-3}&=(0,1), &p_{-2}&=(1,1), &p_{-1}&=(\tfrac{6}{5},\tfrac{9}{10})\end{align*}
has $6$ crossing standardization lines (Figure~\ref{fig:6crossings}). Notice that the symmetry about the $x$-axis makes $\ell_2$ and $\ell_{-2}$ coincide.
\end{observation}

\begin{figure}[htpb]
 \centering
 \includegraphics[width=.6\linewidth]{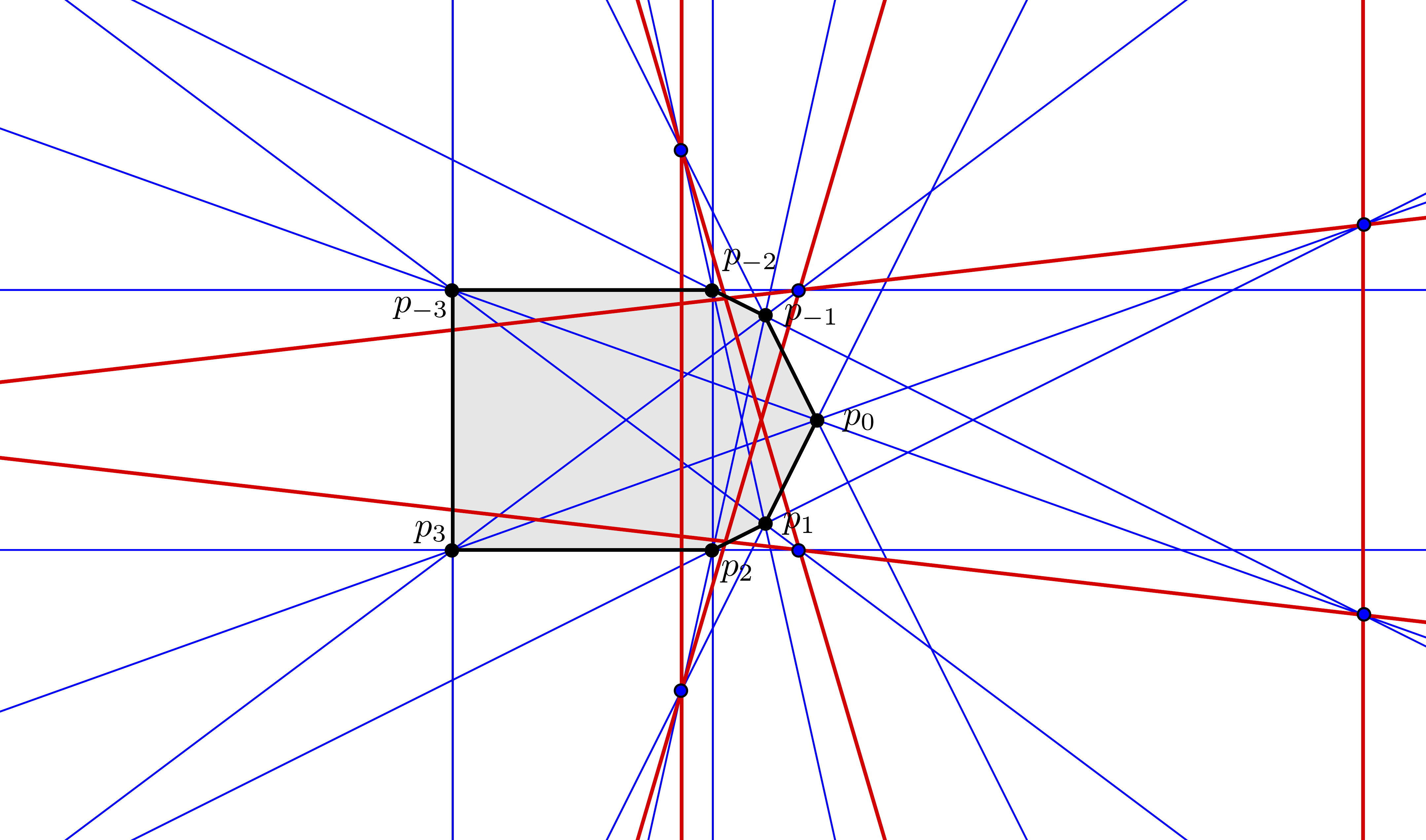}
 \caption{A heptagon with $6$ crossing standardization
lines.}\label{fig:6crossings}
\end{figure}

\subsection{The intersection complexity of heptagons}

Using projective transformations to standardize heptagons is the last step towards
Theorem~\ref{thm:icheptagon}.

\begin{lemma}\label{lem:projective}
  Projective equivalence preserves intersection complexity.
\end{lemma}
\begin{proof}
Let $\sigma:P_1\to P_2$ be a projective transformation between $k$-dimensional
polytopes. Let $Q_1\subset\RR^d$ be a polytope with $\ic(P_1)$ many vertices and
let~$H$ be an affine $k$-flat such that $Q_1\cap H=P_1$. Finally, let $\tau$ be a
projective transformation of~$\RR^d$ that leaves invariant both $H$ and its
orthogonal complement, and such that $\tau|_H=\sigma$. Then $\tau(Q_1)\cap
H=\sigma (P_1)=P_2$.
\end{proof}

\begin{lemma}\label{lem:icheptagon}
 Any heptagon $P$ is a section of a $3$-polytope with no more than $6$ vertices.
\end{lemma}
\begin{proof}
 Let $P$ be a heptagon. By Proposition~\ref{prop:noncrossingstandardization}
it has a non-crossing standardization line, which implies that $P$ is
projectively equivalent to a standard heptagon by
Lemma~\ref{lem:converttostandard}. 
Our claim follows by combining Lemma~\ref{lem:projective} with
Proposition~\ref{prop:icstandardheptagon}.
\end{proof}

The combination of this lemma with the lower bound of Proposition~\ref{prop:ic3bound}
finally yields our claimed result.

\begin{theorem}\label{thm:icheptagon}
Every heptagon has intersection complexity $6$.
\end{theorem}

\section{The intersection complexity of $n$-gons}
We can use Lemma~\ref{lem:icheptagon} to derive bounds for the complexity of
arbitrary polygons. We begin with a trivial bound that presents $n$-gons as sections of $3$-polytopes.

\begin{theorem}\label{thm:ic3ngon}
 Any $n$-gon $P$ with $n\geq 7$ is a section of a $3$-polytope with at most $n-1$
vertices. 
\end{theorem}
\begin{proof}
The proof is by induction. The case $n=7$ is Lemma~\ref{lem:icheptagon}. For
$n\ge8$, let $x=(a,b)$ be a vertex of $P$, and consider the $(n-1)$-gon $P'$
obtained by taking the convex hull of the remaining vertices of $P$. By
induction there is a $3$-polytope  $Q'$ with at most $n-2$ vertices such that
$Q'\cap H_0=P'\times \{0\}$, where $H_0=\set{(x,y,z)\in\RR^3}{z=0}$. Then
$Q=\conv\big(Q'\cup (a,b,0)\big)$ satisfies $Q\cap H_0=P\times \{0\}$, and has $n-1$ vertices.
\end{proof}

\begin{question}
 Which is the smallest $f(n)$ such that any $n$-gon is a section of a $3$-polytope with at most $f(n)$ vertices? Is $f(n)\sim \frac23 n$?
\end{question}

We can derive more interesting bounds when we allow ourselves to increase the
dimension. We only need the following result (compare \cite[Proposition~2.8]{ThomasParriloGouveia2013}).
\begin{lemma}\label{lem:icunion}
 Let $P_1$ and $P_2$ be polytopes in $\RR^d$, and let $P=\conv(P_1\cup P_2)$. 
 If $P_i$ is a section of a $d_i$-polytope with $n_i$ vertices for $i=1,2$, then $P$ is a section of a $(d_1+d_2-d)$ polytope with not more than $n_1+n_2$ vertices. In particular,  $\ic(P)\leq \ic(P_1)+\ic(P_2)$.
\end{lemma}
\begin{proof}
For $i=1,2$, let $Q_i$ be a polytope in $\RR^{d_i}$ with $n_i$ vertices 
and such that $Q_i\cap H_i=P_i$,
where $H_i=\set{x\in \RR^d}{x_j =0 \text{ for }d_i-d< j\leq d}$ is the
$d$-flat that contains the points with vanishing last $d_i-d$ coordinates.
Now consider the following embeddings of $Q_1$ and $Q_2$ in $\RR^{d_1+d_2-d}$:
\begin{itemize}
 \item for $q\in Q_1$ let $f_1(q)=(q_1,\dots,q_{d},q_{d+1},\dots,q_{d_1},
0,\dots,0)$, and
 \item for $q\in Q_2$ let
$f_2(q)=(q_1,\dots,q_{d},0,\dots,0,q_{d+1},\dots,q_{d_2})$.
\end{itemize}
Finally, consider the polytope $Q:=\conv\big(f_1(Q_1)\cup f_2(Q_2)\big)$,
which has at most $n_1+n_2$ vertices, and the $d$-flat $H:=\set{x\in\RR^{d_1+d_2-d}}{x_j=0 \text{ for }d<j\leq d_1+d_2-d}$;
then $P=Q\cap H$.
\end{proof}

\begin{theorem}\label{thm:icngon}
Any $n$-gon with $n\geq 7$ is a section of a $(2+\ffloor{n}{7})$-dimensional
polytope with at most $\fceil{6n}{7}$ vertices. In particular, $\ic(P)\leq
\fceil{6n}{7}$.
\end{theorem}
\begin{proof}
This is a direct consequence of Lemmas~\ref{lem:icheptagon}
and~\ref{lem:icunion}.
\end{proof}

\begin{question}
 Which is the smallest $f(n)$ such that any $n$-gon is a section of a $g(n)$-dimensional polytope with at most $f(n)$ vertices? Are $f(n)=O(\sqrt{n})$ and $g(n)=O(\sqrt{n})$?
\end{question}

\section*{Acknowledgements}
The authors want to thank G\"unter Ziegler and Vincent Pilaud for many enriching discussions on this subject.

\end{document}